\documentclass[a4paper,9pt]{article}
\usepackage{amsmath,amssymb,textcomp,bbm,xcolor,latexsym,theorem}
\usepackage{setspace}

\usepackage[english]{babel}

\usepackage{authblk}

\usepackage{tikz}

\newcommand{\assign}{:=}
\newcommand{\mathd}{\mathrm{d}}
\newcommand{\nobracket}{}
\newcommand{\nocomma}{}
\newcommand{\noplus}{}
\newcommand{\nosymbol}{}
\newcommand{\tmcolor}[2]{{\color{#1}{#2}}}
\newcommand{\tmem}[1]{{\em #1\/}}
\newcommand{\tmmathbf}[1]{\ensuremath{\boldsymbol{#1}}}
\newcommand{\tmname}[1]{\textsc{#1}}
\newcommand{\tmop}[1]{\ensuremath{\operatorname{#1}}}
\newcommand{\tmperson}[1]{\textsc{#1}}
\newcommand{\tmstrong}[1]{\textbf{#1}}
\newcommand{\tmtextit}[1]{{\itshape{#1}}}
\newcommand{\tmtextsc}[1]{{\scshape{#1}}}
\newcommand{\tmtextup}[1]{{\upshape{#1}}}
\newenvironment{descriptionaligned}{\begin{description} }{\end{description}}
\theorembodyfont{\upshape}
\newtheorem{theorem} {\sc  Theorem\rm} [section]

\newtheorem{corollary} [theorem] {\sc  Corollary\rm}

\newtheorem{proposition} [theorem] {\sc  Proposition\rm}

\newtheorem{definition}[theorem]{\sc  Definition\rm}
\newtheorem{proof}[theorem]{\sc  Proof\rm}

\newtheorem{remark}[theorem]{\sc  Remark\rm}

\newcommand{\RR}{\mathbb{R}}
\newcommand{\NN}{\mathbb{N}}
\newcommand{\SSbb}{\mathbb{S}}

\begin{document}

\title{Homogenization of Composite Ferromagnetic Materials}


\author[1]{Fran{\c c}ois Alouges \thanks{francois.alouges@polytechnique.edu}}
\author[1]{Giovanni Di Fratta\thanks{difratta@cmap.polytechnique.fr}}
\affil[1]{CMAP, \'Ecole Polytechnique, route de Saclay, 91128 Palaiseau Cedex, FRANCE}

\normalfont
\spaceskip=.9\fontdimen2\font
      plus .9\fontdimen3\font
     minus .9\fontdimen4

\begin{abstract}
Nowadays, nonhomogeneous and {\tmem{periodic}} ferromagnetic materials are the
subject of a growing interest. Actually such periodic configurations often
combine the attributes of the constituent materials, while sometimes, their
properties can be strikingly different from the properties of the different
constituents. These periodic configurations can be
therefore used to achieve physical and chemical properties difficult to
achieve with homogeneous materials. To predict the magnetic behavior of such
composite materials is of prime importance for
applications. 

The main objective of this paper is to perform, by means of {\tmem{$\Gamma$-convergence}}
and {\tmem{two-scale
convergence}}, 
a  rigorous derivation of the
  homogenized {\tmname{Gibbs-Landau}} free energy functional associated to a composite
periodic ferromagnetic material, i.e. a ferromagnetic material in which the
heterogeneities are periodically distributed inside the ferromagnetic media.
 We  thus describe the $\Gamma$-limit of the \tmtextsc{Gibbs-Landau} free energy
  functional, as the period over which the heterogeneities are distributed
  inside the ferromagnetic body shrinks to zero.
\end{abstract}

\section{Introduction}

Composite materials are an important class of natural or engineered
heterogeneous media, composed of a mixture of two or more constituents with
significantly different physical or chemical properties, firmly bonded
together, which remain separate and distinct within the finished structure.
Finding a model which considers the composite as a bulk and whose coefficients
and terms are computed from suitable averages of those of its constituents and
the geometry of the microstructure is the aim of homogenization theory. The
study of composites and their homogenization is a subject with a long history,
which has attracted the interest and the efforts of some of the most
illustrious names in science~{\cite{Markov1999}}: In 1824, {\tmname{Poisson}},
in his first {\tmem{M{\'e}moire sur la th{\'e}orie du
magn{\'e}tisme}}~{\cite{Poisson1826}}, put the basis of the theory of induced
magnetism assuming a model in which the body is composed of conducting spheres
embedded in a nonconducting material. This paper is the origin of the basic
models and ideas that prevailed in the theory of heterogeneous media in almost
all domains of continuum mechanics, for almost a century after its
appearance~{\cite{Markov1999}}. We refer the reader to the papers of
{\tmname{Markov}}~{\cite{Markov1999}} and
{\tmname{Landauer}}~{\cite{Landauer1978}} for more historical details.

Nowadays, nonhomogeneous and {\tmem{periodic}} ferromagnetic materials are the
subject of a growing interest. Actually such periodic configurations often
combine the attributes of the constituent materials, while sometimes, their
properties can be strikingly different from the properties of the different
constituents~{\cite{milton2002theory}}. These periodic configurations can be
therefore used to achieve physical and chemical properties difficult to
achieve with homogeneous materials. To predict the magnetic behavior of such
composite materials is of prime importance for
applications~{\cite{milton2002theory}}. From a mathematical point of view, the
study of composite materials, and more generally of media which involve
microstructures, is the main source of inspiration for the {\tmem{Mathematical
Theory of Homogenization}} which, roughly speaking, is a mathematical
procedure which aims at understanding heterogeneous materials with highly
oscillating heterogeneities (at the microscopic level) via an effective
model~{\cite{Nandakumaran2007}}.

The main objective of this paper is to perform, in the framework of
{\tmname{De Giorgi}}'s notion of {\tmem{$\Gamma$-convergence}}
{\cite{DeGiorgi1975}} and {\tmname{Allaire}}'s notion of {\tmem{two-scale
convergence}} {\cite{Allaire1992}} (see also the paper by {\tmname{Nguetseng}}
{\cite{Nguetseng1989}}), a mathematical homogenization study of the
{\tmname{Gibbs-Landau}} free energy functional associated to a composite
periodic ferromagnetic material, i.e. a ferromagnetic material in which the
heterogeneities are periodically distributed inside the ferromagnetic media.
Compared to earlier works related to the subject (see for instance {\cite{DeSimone2,DeSimone1,Haddar,Santugini2007}}) we consider here the full {\tmname{Gibbs-Landau}} functional for mixtures of different materials in the three-dimensional space.

\subsection{The Landau-Lifshitz micromagnetic theory of single-crystal
ferromagnetic materials}

According to {\tmname{Landau}} and {\tmname{Lifshitz}} micromagnetic theory of
ferromagnetic materials (see {\cite{Bertotti1998,Brown1,Landau1935,Mayergoyz}}), the states
of a rigid {\tmem{single-crystal}} ferromagnet, occupying a region $\Omega
\subseteq \RR^3$, and subject to a given external magnetic field
$\tmmathbf{h}_a$, are described by a vector field, the magnetization
$\mathbf{M}$, verifying the so-called {\tmem{fundamental constraint}}
{\tmem{of micromagnetic theory}}: A ferromagnetic body is always locally
saturated, i.e. there exists a positive constant $M_s$ such that
\begin{equation}
  | \mathbf{M} | = M_s (T) \text{ \ a.e. in } \Omega . \label{eq:fcmicromag}
\end{equation}
The {\tmem{saturation magnetization}} $M_s$ depends on the specific material
and on the temperature $T$, and vanishes above a temperature (characteristic
of each crystal type) known as the {\tmname{Curie}} point. Since we will
assume that the specimen is at a fixed temperature below the {\tmname{Curie}}
point of the material, the value $M_s$ will be regarded as a material
dependent function, and therefore as a constant function when working on
single-crystal ferromagnets. Due to the constraint (\ref{eq:fcmicromag}) in
the sequel we express the magnetization $\mathbf{M}$ under the form
$\mathbf{M} \assign M_s (T) \tmmathbf{m}$ where $\tmmathbf{m}: \Omega
\rightarrow \SSbb^2$ is a vector field which takes its values on the unit
sphere $\SSbb^2$ of $\RR^3$.

Even though the magnitude of the magnetization vector is constant in space, in
general it is not the case for its direction, and the observable states can be
mathematically characterized as local minimizers of the
{\tmname{Gibbs-Landau}} free energy functional associated to the
single-crystal ferromagnetic particle (using the notation of \cite{Bertotti1998,Mayergoyz})
\begin{align}
       \mathcal{G}_{\mathcal{L}} (\tmmathbf{m}) \assign & \underset{= : \mathcal{E}
  (\tmmathbf{m})}{\int_{\Omega} a_{\tmop{ex}} | \nabla \tmmathbf{m} |^2 \,\mathd
  \tau} + \underset{= : \mathcal{A} (\tmmathbf{m})}{\int_{\Omega}
  \varphi_{\tmop{an}} (\tmmathbf{m}) \, \mathd \tau} \nonumber\\ 
	&\quad\underset{= : \mathcal{W}
  (\tmmathbf{m})}{- \frac{{\mu}_0}{2} \int_{\Omega}
  \tmmathbf{h}_{\text{d}} [M_s \tmmathbf{m}] \cdot M_s \tmmathbf{m} \,\mathd
  \tau}  \underset{= : \mathcal{Z} (\tmmathbf{m})}{-{\mu}_0 \int_{\Omega}
  \tmmathbf{h}_a \cdot M_s \tmmathbf{m}\, \mathd \tau}\; .  \label{eq:GLunorm}
\end{align}
The first term, $\mathcal{E} (\tmmathbf{m})$, called {\tmem{exchange
energy}}, penalizes spatial variations of $\tmmathbf{m}$. The factor
$a_{\tmop{ex}}$ in the term is a phenomenological positive material constant
which summarizes the effect of (usually very) short-range exchange
interactions.

The second term, $\mathcal{A} (\tmmathbf{m})$, or the {\tmem{anisotropy
energy}}, models the existence of preferred directions for the magnetization
(the so-called {\tmem{easy axes}}), which usually depend on the crystallographic structure of the material.
The anisotropy energy density
$\varphi_{\tmop{an}} : \SSbb^2 \rightarrow \RR^+$ is assumed to be a
non-negative even and globally lipschitz continuous function, that vanishes
only on a finite set of unit vectors (the {\tmem{easy axes}}).

The third term, $\mathcal{W} (\tmmathbf{m})$, is called the
{\tmem{magnetostatic self-energy}}, and is the energy due to the (dipolar)
magnetic field, also known in literature as the stray field,
$\tmmathbf{h}_{\text{d}} [\tmmathbf{m}]$ generated by $\tmmathbf{m}$. From the
mathematical point of view, assuming $\Omega$ to be open, bounded and with a
Lipschitz boundary, a given magnetization $\tmmathbf{m} \in L^2 \left( \Omega,
\mathbb{\RR}^3 \right)$ generates the stray field $\tmmathbf{h}_{\mathd}
[\tmmathbf{m}] = \nabla u_{\tmmathbf{m}}$ where the potential
$u_{\tmmathbf{m}}$ solves:
\begin{equation}
  \Delta u_{\tmmathbf{m}} = - \tmop{div} (\tmmathbf{m} \chi_{\Omega}) \text{ \
  in \ } \mathcal{D}^{\prime} \left( \RR^3 \right) . \label{eq:potential}
\end{equation}
In (\ref{eq:potential}) we have indicated with $\tmmathbf{m} \chi_{\Omega}$
the extension of $\tmmathbf{m}$ to $\RR^3$ that vanishes outside $\Omega$.
{\tmname{Lax-Milgram}} theorem guarantees that equation (\ref{eq:potential})
possesses a unique solution in the {\tmname{Beppo-Levi}} space:
\begin{equation}
  BL^1\left(\mathbb{\RR}^3 \right) = \left\{ u \in \mathcal{D}^{\prime}
  ( \RR^3 ) \, : \, \frac{u (\cdot)}{\sqrt{1 + | \cdot |^2}} \in
  L^2 ( \mathbb{\RR}^3) \text{, \ } \nabla u \in L^2 (
  \mathbb{\RR}^3, \mathbb{\RR}^3 ) \right\} .
\end{equation}
Eventually, the fourth term $\mathcal{Z} (\tmmathbf{m})$, is called the
{\tmem{interaction energy}} (or \tmem{Zeeman energy}), and models the tendency of a specimen to have its
magnetization aligned with the external field $\tmmathbf{h}_a$, assumed to be
unaffected by variations of $\tmmathbf{m}$.

The competition of those four terms explain most of the striking pictures of
the magnetization that ones can see in most ferromagnetic material
{\cite{Hubert1998}}, in particular the so-called {\tmem{domain structure}},
that is large regions of uniform or slowly varying magnetization (the
{\tmem{magnetic domains}}) separated by very thin transition layers (the
{\tmem{domain walls}}).

\subsection{The Gibbs-Landau energy functional associated to composite
ferromagnetic materials}

Physically speaking, when considering a ferromagnetic body composed of several
magnetic materials (i.e. a non single-crystal ferromagnet) a new mathematical
model has to be introduced. In fact, as far as the ferromagnet is no more a
single crystal, the material depending functions $a_{\tmop{ex}}, M_s (T)$ and
$\varphi_{\tmop{an}}$ are no longer constant on the region $\Omega$ occupied
by the ferromagnet. Moreover one has to describe the local interactions of two
grains with different magnetic properties at their touching interface
{\cite{Acerbi2004}}.

From a mathematical point of view, this latter requirement is usually taken
into account in two different ways. Either one adds to the model a surface
energy term which penalizes jumps of the magnetization direction
$\tmmathbf{m}$ at the interface of both grains, or, and we stick on this later
on, one simply considers a {\tmem{strong coupling}}, meaning that the
direction of the magnetization does not jump through an interface. We insist
on the fact that only the direction is continuous at an interface while the
magnitude $M_s$ $\tmop{is} \tmop{obviously} \tmop{discontinuous} \nosymbol .$
Therefore, the natural mathematical setting for the problem turns out to be
characterized by the assumption that the magnetization direction
$\tmmathbf{m}$ is in the ``weak'' {\tmname{Sobolev}} metric space $\left( H^1
( \Omega, \SSbb^2 ), d_{L^2 ( \Omega, \SSbb^2 )}
\right)$, i.e. on the metric subspace of $H^1 ( \Omega, \RR^3 )$
constituted by the functions constrained to take values on the unit sphere of
$\RR^3$ and endowed with the $L^2 (\Omega)$ metric. It is in this framework
that we will conduct our work from now on.  
	\begin{figure}
		\centering
			\includegraphics[width=0.6\textwidth]{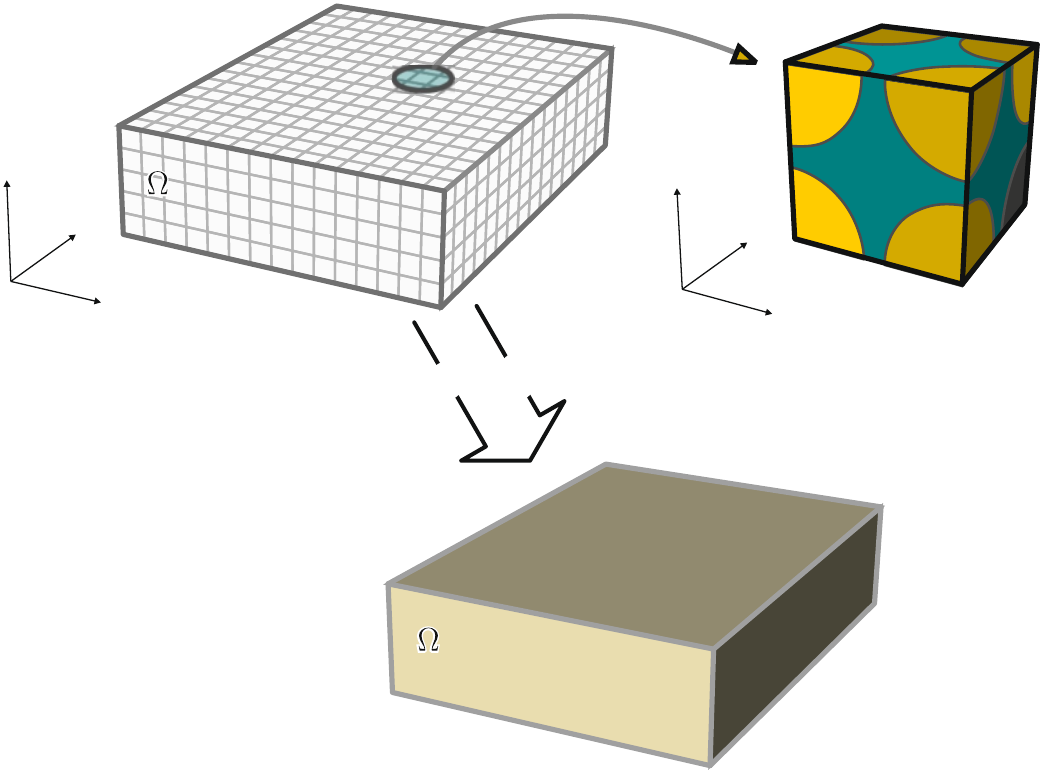} 
		 \caption{\label{fig1}If we assume that the heterogeneities are evenly
  distributed inside the ferromagnetic media $\Omega$, we can model the
  material as periodic. As illustrated in the figure, this means that we can
  think of the material as being built up of small identical cubes
  $Q_{\varepsilon}$, the side length of which we call $\varepsilon$.}
	\end{figure}
We start by recalling the basic idea of the mathematical theory of
homogenization. Let $\Omega \subset \RR^3$ be the region occupied by the
composite material. If we assume that the heterogeneities are regularly
distributed, we can model the material as periodic. As illustrated in
Fig.\ref{fig1}, this means that we can think of the material as being built up
of small identical cubes, the side length of which being called $\varepsilon$.
Let $Q = [0, 1]^3$ be the unit cube of $\RR^3$. We let for $y \in Q \text{, }
a_{\tmop{ex}} (y), M_s (y), \varphi_{\tmop{an}} (y, \tmmathbf{m})$ be the
periodic repetitions of the functions that describe how the exchange constant
$a_{\tmop{ex}}$, the saturation magnetization $M_s$ and the anisotropy density
energy $\varphi_{\tmop{an}} (y, \tmmathbf{m})$ vary over the representative
cell $Q$ (see Fig. \ref{fig1}). Substituting $x / \varepsilon$ for $y$, we
obtain the {\guillemotleft}two-scale{\guillemotright} functions
$a_{\varepsilon} (x) \assign a_{\tmop{ex}} (x / \varepsilon), M_{\varepsilon}
(x) \assign M_s (x / \varepsilon)$ and $\varphi_{\varepsilon} (x,
\tmmathbf{m}) \assign \varphi_{\tmop{an}} (x / \varepsilon, \tmmathbf{m})$
that oscillate periodically with period $\varepsilon$ as the variable $x$ runs
through $\Omega$, describing the oscillations of the material dependent
parameters of the composite. At every scale $\varepsilon$, the energy
associated to the $\varepsilon$-heterogeneous ferromagnet, will be given by
the following generalized {\tmname{Gibbs}}-{\tmname{Landau}} energy functional
\begin{align}
  \mathcal{G}_{\mathcal{L}}^{\varepsilon} (\tmmathbf{m}) \assign &\underset{= :
  \mathcal{E}_{\varepsilon} (\tmmathbf{m})}{\int_{\Omega} a_{\varepsilon} |
  \nabla \tmmathbf{m} |^2 \,\mathd \tau} + \underset{= :
  \mathcal{A}_{\varepsilon} (\tmmathbf{m})}{\int_{\Omega}
  \varphi_{\varepsilon} (\cdot, \tmmathbf{m}) \,\mathd \tau}  \nonumber\\
	&\quad\underset{= :\mathcal{W}_{\varepsilon} (\tmmathbf{m})}{- \frac{{\mu}_0}{2}
  \int_{\Omega} \tmmathbf{h}_{\text{d}} [M_{\varepsilon} \tmmathbf{m}] \cdot
  M_{\varepsilon} \tmmathbf{m} \,\mathd \tau}  \underset{= :
  \mathcal{Z}_{\varepsilon} (\tmmathbf{m})}{-{\mu}_0 \int_{\Omega}
  \tmmathbf{h}_a \cdot M_{\varepsilon} \tmmathbf{m} \,\mathd \tau} 
  \label{eq:GLComposites} .
\end{align}

The asymptotic $\Gamma$-convergence analysis of the family of functionals
$(\mathcal{G}_{\mathcal{L}}^{\varepsilon})_{\varepsilon \in \RR^+}$ as
$\varepsilon$ tends to 0, is the object of the present paper.

\subsection{Statement of the main result}

The main purpose of this paper is to analyze, by the means of both
$\Gamma$-convergence and two-scale convergence techniques, the asymptotic
behavior, as $\varepsilon \rightarrow 0$, of the family of Gibbs-Landau free
energy functionals $(\mathcal{G}_{\mathcal{L}}^{\varepsilon})_{\varepsilon \in
\RR^+}$ expressed by (\ref{eq:GLComposites}). Let us make the statement more
precise.

We consider the unit sphere $\SSbb^2$ of $\mathbbm{R}^3$ and, for every $s
\in \SSbb^2$, the tangent space of $\SSbb^2$ at a point $s$ will be denoted by
$T_s \left( \SSbb^2 \right)$. The class of admissible maps we are interested
in is defined as
\[ H^1 ( \Omega, \SSbb^2 ) \assign \left\{ \tmmathbf{m} \in H^1
   ( \Omega, \RR^3 ) \; : \; \tmmathbf{m} (x) \in \SSbb^2 \text{
   for $\tau$-a.e. } x \in \Omega \right\}, \]
where we have denoted by $\tau$ the Lebesgue measure on $\RR^3$. 
We consider
$H^1 ( \Omega, \SSbb^2 )$ as a metric space endowed with the metric
structure induced by the classical $L^2 ( \Omega, \RR^3 )$ metric.
For every positive real number $t > 0$, we set $Q_t \assign [0, t]^3$ and $Q
\assign Q_1 = [0, 1]^3$. We recall that a function $u : \RR^3 \rightarrow \RR$
is said to be $Q$-periodic if $u (\cdot) = u (\cdot + e_i)$ for every $e_i$ in
the canonical basis $(e_1, e_2, e_3)$ of $\RR^3$. 

For the energy densities
appearing in the family
$(\mathcal{G}_{\mathcal{L}}^{\varepsilon})_{\varepsilon \in \RR^+}$ we assume
the following hypotheses:

\begin{descriptionaligned}
  \item[$\tmmathbf{[H_1]}$] The {\tmem{exchange parameter}} $a_{\tmop{ex}}$ is
  supposed to be a $Q$-periodic measurable function belonging to $L^{\infty}
  (Q)$ which is bounded from below and above by two positive constants
  $c_{\tmop{ex}} > 0, C_{\tmop{ex}} > 0$, i.e. $0 < c_{\tmop{ex}} \leqslant
  a_{\tmop{ex}} (y) \leqslant C_{\tmop{ex}}$ for $\tau$-a.e. $y \in Q$. In the
  setting of classical Calculus of Variations, this hypothesis guarantees that
  the exchange energy density, which has the form $g (y, \xi) \assign
  a_{\tmop{ex}} (y) | \xi |^2$, $\xi \in \RR^{3 \times 3}$, is a
  {\tmname{Carath{\'e}odory}} integrand satisfying the following quadratic
  growth condition for $\tau$-a.e. $y \in Q$
  \begin{equation}
    \forall \xi \in \RR^{3 \times 3} \hspace{1em} c_{\tmop{ex}} | \xi |^2
    \leqslant g (y, \xi) \leqslant C_{\tmop{ex}} (1 + | \xi |^2) .
  \end{equation}
  Then we set $a_{\varepsilon} (x) \assign a_{\tmop{ex}} (x / \varepsilon)$.
  
  \item[$\tmmathbf{[H_2]}$] The {\tmem{anisotropy density energy}}
  $\varphi_{\tmop{an}} : \RR^3 \times \SSbb^2 \rightarrow \RR^+$ is supposed
  to be a $Q$-periodic measurable function belonging to $L^{\infty} (Q)$ with
  respect to the first variable, and globally lipschitz with respect to the
  second one (uniformly with respect to the first variable), i.e. $\exists
  \kappa_L > 0$ such that
  \begin{equation}
    _{} \underset{y \in Q}{\tmop{ess} \sup}  | \varphi_{\tmop{an}} (y, s_1) -
    \varphi_{\tmop{an}} (y, s_2) | \leqslant \kappa_L | s_1 - s_2 |
    \hspace{1em} \forall s_1, s_2 \in \SSbb^2 .
  \end{equation}
  We then set $\varphi_{\varepsilon} (x, s) \assign \varphi_{\tmop{an}} (x /
  \varepsilon, s)$. The hypotheses assumed on $\varphi_{\tmop{an}}$ are
  sufficiently general to treat the most common classes of crystal anisotropy
  energy densities arising in applications. As a sake of example, for uniaxial
  anisotropy, the energy density reads as
  \begin{equation}
    \varphi_{\tmop{an}} (y, \tmmathbf{m} (x)) = \kappa (y) [1 - (\tmmathbf{u}
    (y) \cdot \tmmathbf{m} (x))^2],
  \end{equation}
  the spatially dependent unit vector $\tmmathbf{u} (\cdot)$ being the easy
  axis of the crystal. For cubic type anisotropy, the energy density reads as:
  \begin{equation}
    \varphi_{\tmop{an}} (y, \tmmathbf{m} (x)) = \kappa (y) \sum_{i = 1}^3
    [(\tmmathbf{u}_i (y) \cdot \tmmathbf{m} (x))^2 - (\tmmathbf{u}_i (y) \cdot
    \tmmathbf{m} (x))^4],
  \end{equation}
  the mutually orthogonal unit vectors $\tmmathbf{u}_i (\cdot)$ being the
  three {\tmem{easy-axes}} of the cubic crystal. Note that the anisotropy depends on the material both in strength 
  ($\kappa (y)$) and in direction ($\tmmathbf{u}_i (y)$).
  \item[$\tmmathbf{[H_3]}$] The {\tmem{saturation magnetization}} $M_s$ is
  supposed to be a $Q$-periodic measurable function belonging to $L^{\infty}
  (Q)$, and we set $M_{\varepsilon} (\cdot) = M_s (\cdot / \varepsilon)$.
\end{descriptionaligned}

The main result of this paper is the following:

\begin{theorem}
  \label{thm:mainresult} Let
  $(\mathcal{G}_{\mathcal{L}}^{\varepsilon})_{\varepsilon \in \RR^+}$ be a
  family of {\tmname{Gibbs-Landau}} free energy functionals satisfying the
  hypotheses $\tmmathbf{[H_1]}, \tmmathbf{[H_2]}$ and $\tmmathbf{[H_3]}$. Then
  $(\mathcal{G}_{\mathcal{L}}^{\varepsilon})_{\varepsilon \in \RR^+}$ is
  {\tmstrong{equicoercive}} in the metric space $\left( H^1 ( \Omega,
  \SSbb^2 ), d_{L^2 ( \Omega, \SSbb^2 )} \right)$. Moreover
  $(\mathcal{G}_{\mathcal{L}}^{\varepsilon})_{\varepsilon \in \RR^+}$
  $\Gamma$-converges in $\left( H^1 ( \Omega, \SSbb^2 ), d_{L^2
  ( \Omega, \SSbb^2 )} \right)$ to the functional
  $\mathcal{G}_{\hom} : H^1 ( \Omega, \SSbb^2 ) \rightarrow \RR^+$
  defined by
  \begin{equation}
    \mathcal{G}_{\hom} (\tmmathbf{m}) \assign \mathcal{E}_{\hom}
    (\tmmathbf{m}) +\mathcal{A}_{\hom} (\tmmathbf{m}) + \frac{{\mu}_0}{2}
    \mathcal{W}_{\hom} (\tmmathbf{m}) +{\mu}_0 \mathcal{Z}_{\hom}
    (\tmmathbf{m}) . \label{eq:homGLEF}
  \end{equation}
  The four terms that appear in (\ref{eq:homGLEF}) have the following
  expressions: The homogenized {\tmstrong{exchange energy}} is given by
  \begin{equation}
    \mathcal{E}_{\hom} (\tmmathbf{m}) \assign \int_{\Omega} A_{\hom} \nabla
    \tmmathbf{m} (x) : \nabla \tmmathbf{m} (x) \,\mathd x,
  \end{equation}
  where $A_{\hom}$ is the {\guillemotleft}classical{\guillemotright}
  homogenized tensor $A_{\hom}$ given by the average
  \begin{equation}
    A_{\hom} \assign \langle a_{\tmop{ex}} (y) (I + \nabla \tmmathbf{\varphi}
    (y))^T (I + \nabla \tmmathbf{\varphi} (y))  \rangle_Q, \;
    \tmmathbf{\varphi} \assign (\varphi_1, \varphi_2, \varphi_3),
  \end{equation}
  where for every $j \in \NN_3$ the component $\varphi_j$ is the unique
  (up to a constant) solution of the following scalar unit cell problem
  \begin{equation}
    \varphi_j \assign \underset{\varphi \in H_{\#}^1
    (Q)}{\tmop{argmin}}_{} \int_Q a_{\tmop{ex}} (y) [ | \nabla \varphi (y) +
    e_j |^2 ] \,\mathd y .
  \end{equation}
  The homogenized {\tmstrong{anisotropy energy}} is given by

  \begin{equation}
    \mathcal{A}_{\hom} (\tmmathbf{m}) \assign \int_{\Omega \times Q}
    \varphi_{\tmop{an}} (y, \tmmathbf{m} (x)) \,\mathd y \, \mathd x,
    \label{eq:homAnthm1}
  \end{equation}
  while the homogenized {\tmstrong{magnetostatic self-energy}} is given by
  \begin{equation}
    \mathcal{W}_{\hom} (\tmmathbf{m}) \assign - \langle M_s \rangle_Q^2
    \int_{\Omega} \tmmathbf{h}_{\mathd} [\tmmathbf{m}] \cdot \tmmathbf{m}
    \,\mathd \tau + \int_{\Omega \times Q} | \nabla_y v_{\tmmathbf{m}} (x, y)
    |^2 \,\mathd x \,\mathd y,
  \end{equation}
  where, for every $x \in \Omega$, the scalar function $v_{\tmmathbf{m}} :
  \Omega \times Q \rightarrow \RR$, is the unique solution of the cell
  problem:
  \begin{equation}
    \tmmathbf{m} (x) \cdot \int_Q M_s (y) \nabla_y \psi (y) \,\mathd y = -
    \int_Q \nabla_y v_{\tmmathbf{m}} (x, y) \cdot \nabla_y \psi (y) \,\mathd y,
    \; \int_Q v_{\tmmathbf{m}} (x, y)\, \mathd y = 0
  \end{equation}
  for all $\psi \in H^1_{\#} (Q)$. 
	
	Finally, the homogenized
  {\tmstrong{interaction energy}} is given by
  \begin{equation}
    \mathcal{Z}_{\hom} (\tmmathbf{m}) = - \langle M_s \rangle_Q \int_{\Omega}
    \tmmathbf{h}_a \cdot \tmmathbf{m} \, \mathd \tau . \label{eq:homIntthm1}
  \end{equation}
\end{theorem}

The paper is organized as follows: In Section \ref{sec:prelim} we give a brief
survey of the main mathematical concepts and results used throughout the
paper. The equicoercivity of the family
$(\mathcal{G}_{\mathcal{L}}^{\varepsilon})_{\varepsilon \in \RR^+}$ is
established in Section \ref{subsec:proof1}; the $\Gamma$-limit of the exchange
energy family of functionals $(\mathcal{E}_{\varepsilon})_{\varepsilon \in
\RR^+}$ is computed in Section \ref{subsec:proof2}; in Section
\ref{subsec:proof3} it is shown that the family of magnetostatic self-energies
$(\mathcal{W}_{\varepsilon})_{\varepsilon \in \RR^+}$ continuously converges
to $\mathcal{W}_{\hom}$, while in Section \ref{subsec:proof4} it is
established the continuous convergence of the family of anisotropy energies
$(\mathcal{A}_{\varepsilon})_{\varepsilon \in \RR^+}$ to $\mathcal{A}_{\hom}$
and the continuous convergence of the family of interaction energies
$(\mathcal{Z}_{\varepsilon})_{\varepsilon \in \RR^+}$ to the functional
$\mathcal{Z}_{\hom}$. Eventually, the proof of {\tmname{mht}} (Theorem
\ref{thm:mainresult}) is completed in Section \ref{sec:proof}, and some well-known 
results in homogenization theory, though somewhat difficult to find in the literature 
are given in the appendix.

\section{Mathematical Preliminaries}\label{sec:prelim}

The purpose of this section is to fix some notations and to give a survey of
the concepts and results that are used throughout this work. All results are
stated without proof as they can be readily found in the references given
below.

\subsection{$\Gamma$-convergence of a family of functionals}

We start by recalling {\tmname{De Giorgi}}'s notion of $\Gamma$-convergence
and some of its basic properties (see {\cite{DeGiorgi1975,DalMaso1993}}).
Throughout this part we indicate with $(X, d)$ a metric space and, for every
$m \in X$, with $\mathfrak{C}_d (m)$ the subset of all sequences of elements
of $X$ which converge to $m$.

\begin{definition}
  {\tmem{\tmtextup{(}$\Gamma$-convergence of a family of
  functionals\tmtextup{)}}} Let $(\mathcal{F}_n)_{n \in \NN}$ be a sequence of
  functionals defined on $X$ with values on $\overline{\RR}$. The functional
  $\mathcal{F}: X \rightarrow \overline{\RR}$ is said to be the
  $\Gamma$-$\lim$ of $(\mathcal{F}_n)_{n \in \NN}$ with respect to the metric
  $d$, if for every $m \in X$ we have:
  \begin{equation}
    \forall (m_n) \in \mathfrak{C}_d (m) \hspace{1em} \mathcal{F} (m)
    \leqslant \liminf_{n \rightarrow \infty} \mathcal{F}_n (m_n)
  \end{equation}
  and
  \begin{equation}
    \exists (\bar{m}_n) \in \mathfrak{C}_d (m) \hspace{1em} \mathcal{F} (m) =
    \lim_{n \rightarrow \infty} \mathcal{F}_n (\bar{m}_n) .
    \label{eq:recovery}
  \end{equation}
  In this case we write $\mathcal{F}= \Gamma \text{-} \lim_{n \rightarrow \infty} \mathcal{F}_n$.
	
	If $(\mathcal{F}_{\varepsilon})_{\varepsilon \in
  \RR^+}$ is a family of functionals, we say that $\mathcal{F}: X \rightarrow \overline{\RR}$
  is the $\Gamma$-$\lim$ of
  $(\mathcal{F}_{\varepsilon})_{\varepsilon \in \RR^+}$ as $\varepsilon \rightarrow 0$, if for every $\varepsilon_n
  \downarrow_{} 0$ one has $\mathcal{F}=\Gamma \text{-} \lim_{n \rightarrow \infty} \mathcal{F}_{\varepsilon_n}$. In this case we write $\mathcal{F}= \Gamma \text{-} \lim_{\varepsilon
  \rightarrow 0} \mathcal{F}_{\varepsilon}$.
	
	The condition (\ref{eq:recovery}) is sometimes referred to in literature as
  the existence of a recovery sequence.
\end{definition}


One of the most important properties of $\Gamma$-convergence, and the reason
why this kind of variational convergence is so important in the asymptotic
analysis of variational problems, is that under appropriate compactness
hypotheses it implies the convergence of (almost) minimizers of a family of
equicoercive functionals to the minimum of the $\Gamma$-limit functional. More
precisely, the following result holds:

\begin{theorem}
  {\tmem{\tmtextup{(}Fundamental Theorem of $\Gamma$-convergence\tmtextup{)}}}
  If $(\mathcal{F}_{\varepsilon})_{\varepsilon \in \RR^+}$ is a family of
  equicoercive functionals $\Gamma$-converging on $X$ to the functional
  $\mathcal{F}$. Then $\mathcal{F}$ is coercive and lower semicontinuous
  (therefore there exists a minimizer for $\mathcal{F}$ on $X$) and we have
  the convergence of minima values
  \begin{equation}
    \min_{m \in X} \mathcal{F} (m) = \lim_{\varepsilon \rightarrow 0} \inf_{m
    \in X} \mathcal{F}_{\varepsilon} (m) .
  \end{equation}
  Moreover, given $\varepsilon_n \downarrow_{} 0$ and $(m_n)_{n \in \NN}$ a
  converging sequence such that
  \begin{equation}
    \lim_{n \rightarrow \infty} \mathcal{F}_{\varepsilon_n} (m_n) = \lim_{n
    \rightarrow \infty}  (\inf_{m \in X} \mathcal{F}_{\varepsilon_n} (m)),
    \label{eq:almostminimizerseq}
  \end{equation}
  its limit is a minimizer for $\mathcal{F}$ on $X$. If
  \tmtextup{(\ref{eq:almostminimizerseq})} holds, the sequence $(m_n)_{n \in
  \NN}$ is said to be a sequence of almost-minimizers for $\mathcal{F}$.
\end{theorem}

Let us recall now that given two families of functional
$(\mathcal{F}_{\varepsilon})_{\varepsilon \in \RR^+}$ and
$(\mathcal{G}_{\varepsilon})_{\varepsilon \in \RR^+}$ $\Gamma$-converging
respectively to $\mathcal{F}$ and $\mathcal{G}$, it is in general not the case (see {\cite{DalMaso1993}})
that $\Gamma$-$\lim_{\varepsilon \rightarrow 0} (\mathcal{F}_{\varepsilon}
+\mathcal{G}_{\varepsilon}) =\mathcal{F}+\mathcal{G}$ . A sufficient condition
for that property to hold is  that at least one of the two families of
functionals satisfies a stronger type of convergence:

\begin{definition}
  We say that a family of functionals
  $(\mathcal{G}_{\varepsilon})_{\varepsilon \in \RR^+}$ is {\tmem{continuously
  convergent}} in $X$ to a functional $\mathcal{G}: X \rightarrow
  \overline{\RR}$, and we will write $\mathcal{G}_{\varepsilon^{}}
  \xrightarrow{\Gamma_{\tmop{cont}}} \mathcal{G}$, if for every $m_0 \in X$
  \[ \lim_{(m, \varepsilon) \rightarrow (m_0, 0)} \mathcal{G}_{\varepsilon}
     (m) =\mathcal{G} (m_0) . \]
\end{definition}

We then have (see {\cite{DalMaso1993}} for a proof):

\begin{proposition}
  \label{prop:sumGlimit}Let $\mathcal{F}= \Gamma \text{-} \lim_{\varepsilon
  \rightarrow 0} \mathcal{F}_{\varepsilon}$. Suppose that the family of
  functionals $(\mathcal{G}_{\varepsilon})_{\varepsilon \in \RR^+}$
  continuously converges to $\mathcal{G}$, and that
  $\mathcal{G}_{\varepsilon}$ and $\mathcal{G}$ are everywhere finite on $X$.
  Then $\mathcal{G}= \Gamma \text{-} \lim_{\varepsilon \rightarrow 0}
  \mathcal{G}_{\varepsilon}$ and
  \[ \Gamma \text{-} \lim_{\varepsilon \rightarrow 0}
     (\mathcal{F}_{\varepsilon} +\mathcal{G}_{\varepsilon})
     =\mathcal{F}+\mathcal{G}. \]
  In particular if $\mathcal{Z}: X \rightarrow \RR$ is a continuous functional
  then $\Gamma \text{-} \lim_{\varepsilon \rightarrow 0}
  (\mathcal{F}_{\varepsilon} +\mathcal{Z}_{}) =\mathcal{F}+\mathcal{Z}$ and
  $\mathcal{Z}$ is called a continuous perturbation of the $\Gamma$-limit.
\end{proposition}

\subsection{$\tmop{Two}$-scale convergence}

The aim of this section is to present in a schematic way the main properties
of two-scale convergence, a notion that is first due to {\tmname{Nguetseng}} {\cite{Nguetseng1989}},
developed as a methodology by {\tmname{Allaire}} {\cite{Allaire1992}} and further investigated by many
others (see {\cite{Allaire1996}} and references therein for instance).

We denote by $C^{\infty}_{\#} (Q)$ the set of infinitely differentiable real
functions over $\RR^3$ that are $Q$-periodic and define $H^1_{\#} (Q)$ as the
closure of $C^{\infty}_{\#} (Q)$ in $H^1_{\text{loc}} (\Omega)$. Obviously any element of $H^1_{\#} (Q)$ has the same trace on the opposite faces of $Q$.

A generalized version of the {\tmname{Riemann-Lebesgue}} Lemma holds for the weak limit of
rapidly oscillating functions. For a proof we refer the reader to
{\cite{Donato1999}}.

\begin{proposition}
  \label{lemma:genRiemLeb}Let $\Omega \subset \RR^3$ be any open set. Let $1
  \leqslant p < \infty$ and $t > 0$ be a positive real number. Let $u \in L^p
  (Q_t)$ be a $Q_t$-periodic function. Set $u_{\varepsilon} (x) \assign u (x /
  \varepsilon)$ $\tau$-a.e$.$ on $\Omega$. Then, if $p < \infty$, as
  $\varepsilon \rightarrow 0$
  \[ u_{\varepsilon} \rightharpoonup \langle u \rangle_{Q_t} \assign
     \frac{1}{| Q_t |} \int_{Q_t} u \, \mathd \tau \hspace{1em} \text{weakly in }
     L^p (\Omega) . \]
  If $p = \infty$, one has
  \[ u_{\varepsilon} \rightharpoonup \langle u \rangle_{Q_t} \assign
     \frac{1}{| Q_t |} \int_{Q_t} u \, \mathd \tau \hspace{1em}
     \text{weakly$^{\ast}$ in } L^{\infty} (\Omega) . \]
\end{proposition}

\begin{definition}
  Let $\Omega$ be an open set $\Omega \subset \RR^3$, and let
  $(\varepsilon_k)_{k \in \NN}$ be a fixed sequence of positive real numbers
  (when it is clear from the context we will omit the subscript $k$)
  converging to $0$. The sequence of functions $(u_{\varepsilon}) \in L^2
  (\Omega)$ is said to two-scale converge to a limit $u \in L^2 (\Omega \times
  Q)$, if for any function $\varphi \in \mathcal{D} [\Omega ; C_{\#}^{\infty}
  (Q)]$ we have
  \begin{equation}
    \lim_{\varepsilon \rightarrow 0} \int_{\Omega} u_{\varepsilon} (x) \varphi
    (x, x / \varepsilon_{}) \,\mathd x = \int_{\Omega \times Q} u (x, y) \varphi
    (x, y) \,\mathd y \,\mathd x
  \end{equation}

  In this case we write $u_{\varepsilon} \twoheadrightarrow u$. We say that
  $(u_{\varepsilon})_{}$ in $L^2 (\Omega)$ {\tmstrong{strongly}} two-scale
  converges to a limit $u \in L^2 (\Omega \times Q)$ if $u_{\varepsilon}
  \twoheadrightarrow u$ and moreover
  \[ \| u \|_{\Omega \times Q} = \lim_{\varepsilon \rightarrow 0} \|
     u_{\varepsilon} \|_{\Omega} . \]
\end{definition}

The importance of this new notion of convergence relies on the following
compactness results.

\begin{proposition}
  \label{prop:compactnessl2ts}For each bounded sequence $(u_{\varepsilon}) \in
  L^2 (\Omega)$, there exists an $u \in L^2 (\Omega \times Q)$ such that, up
  to a subsequence, $u_{\varepsilon} \twoheadrightarrow u$.
\end{proposition}

Moreover, for bounded sequences in $H^1 (\Omega)$ we have the following
result:

\begin{proposition}
  \label{prop:compactnessSts}Let $(u_{\varepsilon})$ be a sequence in $H^1
  (\Omega)$ that converges weakly to a limit $u \in H^1 (\Omega)$. Then
  $u_{\varepsilon} \twoheadrightarrow u_{}$ and there exists a function $v \in
  L^2 [ \Omega ; H^1_{\#} (Q) / \RR ]$ such that, up to a
  subsequence:
  \[ \nabla u_{\varepsilon} \twoheadrightarrow \nabla u + \nabla_y v. \]
\end{proposition}

Next we recall that if the sequence $(u_{\varepsilon})$ is bounded in $L^2
(\Omega)$, it is possible to enlarge the class of test functions used in the
definition of two-scale convergence.

\begin{proposition}
  Let $(u_{\varepsilon})$ be a bounded sequence in $L^2 (\Omega)$ which
  two-scale converges to $u \in L^2 (\Omega \times Q)$. Then
  \[ \lim_{\varepsilon \rightarrow 0} \int_{\Omega} u_{\varepsilon} (x)
     \varphi (x, x / \varepsilon_{}) \,\mathd x = \int_{\Omega \times Q} u (x,
     y) \varphi (x, y) \,\mathd y \,\mathd x \]
  for every $\varphi \in L^2 [\Omega ; C_{\#} (Q)]$.
\end{proposition}

Finally we recall a simple criteria that permits to
{\guillemotleft}bypass{\guillemotright} the problem concerning the convergence
of the product of two $L^2 (\Omega)$-weakly convergence sequences
(cfr.{\cite{Allaire1992,lukkassen2002two}}).

\begin{proposition}
  \label{Prop:2scalproduct}Let $(u_{\varepsilon})$ and
  $(v_{\varepsilon})$ be sequences in $L^2 (\Omega)$
  that respectively two-scale converge to $u$ and $v$ in $L^2 (\Omega \times
  Q)$. If at least one of them {\tmstrong{strongly}} two-scale converges, then
  \[ u_{\varepsilon} v_{\varepsilon} \twoheadrightarrow uv . \]
  In particular, if $(u_{\varepsilon} v_{\varepsilon})$ is bounded in $L^2
  (\Omega)$, from the previous proposition, we have
  \[ \int_{\Omega} u_{\varepsilon} (x) v_{\varepsilon} (x) \varphi (x, x /
     \varepsilon_{}) \,\mathd x = \int_{\Omega \times Q} u (x, y) v (x, y)
     \varphi (x, y) \,\mathd y \,\mathd x \]
  for every $\varphi \in L^2 [\Omega ; C_{\#} (Q)]$.
\end{proposition}

\section{The equicoercivity of the composite Gibbs-Landau free energy
functionals}\label{subsec:proof1}

This section is devoted to the proof of the equicoercivity of the family of
\tmtextsc{Gibbs-Landau} free energy functionals
$(\mathcal{G}^{\varepsilon}_{\mathcal{L}})_{\varepsilon \in \RR^+}$ expressed
by (\ref{eq:GLComposites}). Equicoercivity has an important role in
homogenization theory. In fact, the metric space in which to work, must be
able to guarantee the equicoercivity of the family of functionals under
consideration, i.e. the validity of the Fundamental Theorem of
$\Gamma$-convergence.

\begin{proposition}
  The family $(\mathcal{G}_{\mathcal{L}}^{\varepsilon})_{\varepsilon \in
  \RR^+}$ of {\tmname{Gibbs-Landau}} energy functionals is equicoercive on the
  metric space $\left( H^1 \left( \Omega, \SSbb^2 \right), d_{L^2 \left(
  \Omega, \SSbb^2 \right)} \right)$.
\end{proposition}

\begin{proof}
  According to the hypotheses $[H_1], [H_2]$ and $[H_3]$, there exist positive
  constants $c_{\tmop{ex}}, C_{\tmop{ex}}, C_s, C_{\tmop{an}} \in \RR^+$ such
  that for every $y\in Q$
  \[  0 < c_{\tmop{ex}} \leqslant a_{\tmop{ex}}
     (y) \leqslant C_{\tmop{ex}} \hspace{1em}, \hspace{1em} 0 \leqslant M_s
     (y) \leqslant C_s \hspace{1em}, \hspace{1em} 0 \leqslant \varphi (y,
     \tmmathbf{m}) \leqslant C_{\tmop{an}} . \]
  Therefore, for every $\varepsilon > 0$ on has
	\begin{equation}
	  c_{\tmop{ex}} \| \nabla \tmmathbf{m} \|_{\Omega}^2 \leqslant
    \mathcal{G}_{\mathcal{L}}^{\varepsilon} (\tmmathbf{m})
     \leqslant  C^{\star} (1 + \| \nabla \tmmathbf{m} \|^2_{\Omega})
	\end{equation}
  with $C^{\star} = \max \left\{ C_{\tmop{ex}}, \frac{{\mu}_0}{2} C_s^2 |
  \Omega |, C_s | \Omega |^{1 / 2} \| \tmmathbf{h}_a \|_{\Omega},
  C_{\tmop{an}}^{} | \Omega | \right\}$.
  
  Now observe that for every constant in space magnetization $\tmmathbf{u}$
  and for every $\varepsilon > 0$ on has
  $\mathcal{G}_{\mathcal{L}}^{\varepsilon} (\tmmathbf{u}) \leqslant
  C^{\star}$; thus
  \begin{equation}
	\inf_{\tmmathbf{m} \in H^1 \left( \Omega, \SSbb^2 \right)}
     \mathcal{G}_{\mathcal{L}}^{\varepsilon} = \inf_{\tmmathbf{m} \in K}
     \mathcal{G}_{\mathcal{L}}^{\varepsilon}. 
		\end{equation}
 To finish, we simply observe that, due  to {\tmname{Rellich--Kondrachov}} theorem, $K$ is a compact
  subset of $\left( H^1 ( \Omega, \SSbb^2 ), d_{L^2 ( \Omega,\SSbb^2 )} \right)$.
\end{proof}

\section{The $\Gamma$-limit of exchange energy functionals
$\mathcal{E}_{\varepsilon}$}\label{subsec:proof2}

The fundamental constraint of micromagnetic theory, i.e. the fact that the
domain of definition of the family $\mathcal{E}_{\varepsilon}$ is a manifold
value Sobolev space, plays a fundamental role in the homogenization process.
In fact, although the unconstrained problem has been fully investigated (see
{\cite{Braides1998,marcellini1978periodic,muller1987homogenization}}), it is
not possible to get full information about the manifold constrained
$\Gamma$-limit by just looking at the unconstrained one. This is due to the
well-known fact (see {\cite{DalMaso1993}}) that if
$(\mathcal{F}_{\varepsilon})$ is a family of functionals defined in the metric
space $X$ and $\Gamma$-converging to $\mathcal{F}$, and $\mathcal{G}= \Gamma
\text{-} \lim [\mathcal{G}_{\varepsilon} \assign \mathcal{F}_{\varepsilon | Y
\nobracket}]$ where $(\mathcal{G}_{\varepsilon}) \assign
(\mathcal{F}_{\varepsilon | Y \nobracket})$ represents the family of functionals
obtained as the restriction of $(\mathcal{F}_{\varepsilon})$ to some (metric)
subspace $Y$ of $X$, then $\mathcal{F}_{| Y \nobracket} \leqslant
\mathcal{G}$. Thus the identification of the manifold constrained
$\Gamma$-limit requires more effort.

\subsection{The tangential homogenization theorem}

In what follows we make use of the following theorem due to
{\tmname{Babadjian}} and {\tmname{Millot}} (see {\cite{Babadjian2009}}) in
which the dependence of the $\Gamma$-limit from the tangent bundle of the
manifold is taken into account via the so-called {\tmem{tangentially
homogenized energy density}}. We state the {\tmstrong{tangential
homogenization theorem}} {\tmname{(thm)}} in a bit less general form which is
adequate for our purposes.

\begin{proposition}
  {\tmname{(thm)}} \label{prop:baba}\tmtextup{{\cite{Babadjian2009}}} Let
  $\mathcal{M}$ be a connected smooth submanifold of $\RR^3$ without boundary
  and $g : \RR^3 \times \RR^{3 \times 3} \rightarrow \RR^+$ be a
  Carath{\'e}odory function such that
  \begin{enumerate}
    \item For every $\xi \in \RR^{3 \times 3}$ the function $g (\cdot, \xi)$
    is $Q$-periodic, i.e. such that if $(e_1, e_2, e_3)$ denotes the canonical
    basis of $\RR^3$, one has
    \[ \forall i \in \NN_3, \forall y \in \RR^3, \forall \xi \in \RR^{3 \times
       3} \hspace{1em} g (y + e_i, \xi) = g (y, \xi) . \]
    \item There exist $0 < \alpha \leqslant \beta < \infty$ such that
    \[ \alpha | \xi |^2 \leqslant g (y, \xi) \leqslant \beta (1 + | \xi |^2)
       \hspace{1em} \text{for a.e. $y \in \RR^3$ and all $\xi \in \RR^{3
       \times 3}$} . \]
  \end{enumerate}
  Then the family
  \begin{equation}
    \mathcal{E}_{\varepsilon} (\tmmathbf{m}) \assign \int_{\Omega} g (x /
    \varepsilon, \nabla \tmmathbf{m}) \mathd \tau_{} \label{eq:energyBM}
  \end{equation}
  defined in the metric space $(H^1 (\Omega, \mathcal{M}), d_{L^2 (\Omega,
  \mathcal{M})})$ $\Gamma$-converges to the functional
  \begin{equation}
    \mathcal{E}_{\hom} (\tmmathbf{m}) \assign \int_{\Omega} T_{} g_{\hom}
    (\tmmathbf{m}, \nabla \tmmathbf{m}) \mathd \tau, \label{eq:energyBMhom}
  \end{equation}
  where for every $s \in \mathcal{M}$ and $\xi \in [T_s (\mathcal{M})]^3$, $Q_t \assign [0, t]^3$,
  \begin{equation}
    T_{} g_{\hom} (s, \xi) = \lim_{t \rightarrow \infty}  \left(
    \inf_{\tmmathbf{\varphi} \in W_0^{1, \infty} (Q_t, T_s (\mathcal{M}))}
    \frac{1}{| Q_t |} \int_{Q_t} g (y, \xi + \nabla \tmmathbf{\varphi} (y))
    \mathd y \right), \label{eq:corrector}
  \end{equation}
  is the {\tmstrong{tangentially homogenized energy density}}.
\end{proposition}

We refer the reader to {\cite{Babadjian2009}} for a more general version and
the proof.

\subsubsection{The role of tangent bundle.}

Let us emphasize why the tangent bundle $[\mathcal{T} (\mathcal{M})]^3 \assign
\sqcup_{s \in \mathcal{M}} [T_s (\mathcal{M})]^3$ plays a role. In order to
understand this, it is convenient to develop a minimizer
$\tmmathbf{m}_{\varepsilon}$ of $\mathcal{E}_{\varepsilon}$ under the
so-called multiscale expansion
\begin{equation}
  \tmmathbf{m}_{\varepsilon} (x) =\tmmathbf{m}_{0^{}} (x) \noplus +
  \varepsilon \tmmathbf{m}_1 \left( x, \frac{x}{\varepsilon} \right) + o
  (\varepsilon), \label{eq:multiscaleexp} \noplus \noplus
\end{equation}
where $\tmmathbf{m}_0 \nocomma, \tmmathbf{m}_1$ are respectively a minimizer of the
$\Gamma$-limit of $\mathcal{E}_{\varepsilon}$ and the null average first order
corrector. Clearly, due to the constraint $\tmmathbf{m}_{\varepsilon}(x) \in \mathcal{M}$ for a.e. $x \in \Omega$, we get
\begin{equation}
  0 \equiv \tmmathbf{n} [\tmmathbf{m}_{\varepsilon}] \cdot \nabla
  \tmmathbf{m}_{\varepsilon}, \label{eq:formaltwoscaleconstraint}
\end{equation}
where we have denoted by $\tmmathbf{n}$ the local normal field defined around
$\tmmathbf{m}_{\varepsilon} (x) \in \mathcal{M}$. By passing to the two-scale
limit in both terms of (\ref{eq:formaltwoscaleconstraint}), we formally reach
the equality $0 \equiv \tmmathbf{n} [\tmmathbf{m}_0] \cdot (\nabla
\tmmathbf{m}_0 + \nabla_y \tmmathbf{m}_1) \equiv \tmmathbf{n} [\tmmathbf{m}_0]
\cdot \nabla_y \tmmathbf{m}_1$, which shows that $\tmmathbf{n} [\tmmathbf{m}_0
(x)] \cdot \tmmathbf{m}_1 (x, y)$ does not depend on $y$. Then, passing to the
average over $Q$ we get $\tmmathbf{m}_1 (x, y) \in T_{\tmmathbf{m}_0 (x)}
(\mathcal{M})$. The rigorous formulation of the previous idea is the object of
the next {\tmname{Proposition}}:

\begin{proposition}
  \label{prop:compactnessStsManifold}Let $\mathcal{M}$ be a connected smooth
  submanifold of $\RR^3$, and let $(\tmmathbf{m}_{\varepsilon})$ be a sequence
  in $H^1 (\Omega, \mathcal{M})$ that converges weakly to a limit
  $\tmmathbf{m} \in H^1 \left( \Omega, \RR^3 \right)$. Then
  \[ \tmmathbf{m} \in H^1 (\Omega, \mathcal{M}) \hspace{1em}
     \text{{and}} \hspace{1em} \tmmathbf{m}_{\varepsilon}
     \twoheadrightarrow \tmmathbf{m}. \]
  Moreover there exists a null average function $\tmmathbf{v} \in L^2 [
  \Omega ; H^1_{\#} (Q) / \RR ]$ such that, up to the extraction of a a subsequence:
  \[ \nabla \tmmathbf{m}_{\varepsilon} \twoheadrightarrow \nabla \tmmathbf{m}+
     \nabla_y \tmmathbf{v} \hspace{1em} \text{{{and}}} \hspace{1em}
     \tmmathbf{v} (x, y) \in T_{\tmmathbf{m}_0 (x)} (\mathcal{M}) \text{ for
     a.e. $(x, y) \in \Omega \times Q$} . \]
\end{proposition}

\begin{proof}
  In view of Proposition \ref{prop:compactnessSts}, we only need to prove that 
 $\tmmathbf{v} (x, y) \in T_{\tmmathbf{m}_0 (x)}
  (\mathcal{M})$ for a.e. $(x, y) \in \Omega \times Q$. To this end, let us
  denote by $\tmmathbf{n} (\tmmathbf{m})$ the normal vector at $\tmmathbf{m}
  \in \mathcal{M}$ and observe that it is sufficient to prove that the scalar
  function $\tmmathbf{n} (\tmmathbf{m} (x)) \cdot \tmmathbf{v} (x, y)$ does
  not depend on the $y$ variable, i.e. that in the language of distributions
  on $\Omega \times Y$ one has
  \begin{equation}
    \int_{\Omega \times Q} [\tmmathbf{n} (\tmmathbf{m} (x)) \cdot \tmmathbf{v}
    (x, y)] \tmop{div}_y \tmmathbf{\varphi} (x, y) \mathd y \mathd x = 0
    \hspace{2em} \forall \tmmathbf{\varphi} \in \mathcal{D} [\Omega ;
    C^{\infty}_{\#} (Q)] . \label{eq:indipendencefromy}
  \end{equation}
  Indeed, as far as $\tmmathbf{n} (\tmmathbf{m} (x)) \cdot \tmmathbf{v} (x,
  y)$ is independent from the $y$ variable, since by assumption $\langle
  \tmmathbf{v} (x, \cdot) \rangle_Q = 0$ for a.e. $x \in \Omega$, one has
  $\tmmathbf{n} (\tmmathbf{m} (x)) \cdot \tmmathbf{v} (x, y) =\tmmathbf{n}
  (\tmmathbf{m} (x)) \cdot \langle \tmmathbf{v} (x, \cdot) \rangle_Q = 0$ and
  therefore $\tmmathbf{v} (x, y) \in T_{\tmmathbf{m} (x)} (\mathcal{M})$ for
  a.e. $(x, y) \in \Omega \times Q$.
  
  To prove (\ref{eq:indipendencefromy}) we note that since
  $\tmmathbf{m}_{\varepsilon} \rightarrow \tmmathbf{m}$ in $L^2 (\Omega)$, one
  also has $\tmmathbf{n} (\tmmathbf{m}_{\varepsilon}) \rightarrow \tmmathbf{n}
  (\tmmathbf{m})$ in $L^2 (\Omega)$. Therefore the family $\tmmathbf{n}
  (\tmmathbf{m}_{\varepsilon})$ strongly two-scale converges to
  ${\tmmathbf{n}}\left({\tmmathbf{m}}\right)$ and moreover $\tmmathbf{0}=
  (\tmmathbf{n} (\tmmathbf{m}_{\varepsilon}) \cdot \nabla,
  \tmmathbf{m}_{\varepsilon})_{\Omega} \rightarrow (\tmmathbf{n}
  (\tmmathbf{m}) \cdot \nabla, \tmmathbf{m})_{\Omega}$. Hence, due to
  Proposition \ref{Prop:2scalproduct} we get
\begin{align*}
       0 &= \lim_{\varepsilon \rightarrow 0} \int_{\Omega}
       [\tmmathbf{m}_{\varepsilon} (x) \cdot \nabla \tmmathbf{m}_{\varepsilon}
       (x)] \cdot \tmmathbf{\varphi} (x, x / \varepsilon) \mathd x \nonumber\\ 
			& = \int_{\Omega \times Q} \tmmathbf{m} (x) \cdot [\nabla \tmmathbf{m} (x)
       + \nabla_y \tmmathbf{v} (x, y)] \cdot \tmmathbf{\varphi} (x, y) \, \mathd
       y \mathd x \nonumber\\
       &  =  \int_{\Omega \times Q} [\tmmathbf{m} (x) \cdot \nabla_y
       \tmmathbf{v} (x, y)] \cdot \tmmathbf{\varphi} (x, y) \, \mathd y \mathd
       x \nonumber\\
       & = - \int_{\Omega \times Q} [\tmmathbf{m} (x) \cdot \tmmathbf{v}
       (x, y)] \tmop{div}_y \tmmathbf{\varphi} (x, y) \, \mathd x,
     \end{align*}
  i.e. the desired relation (\ref{eq:indipendencefromy}).
\end{proof}

\subsubsection{The role of tangentially homogenized energy density}

Although the following considerations are not completely rigorous, they give a very good
explanation of the idea behind the expression of the tangentially homogenized
energy density (\ref{eq:corrector}); they are mainly based on the notion of
two-scale convergence while {\cite{Babadjian2009}} relies on $\Gamma$-convergence.

Let us denote by $(\mathcal{E}_{\mathcal{M}}^{\varepsilon})$ the family of
functionals
\begin{equation}
  \tmmathbf{m} \in (H^1 (\Omega, \mathcal{M}), d_{L^2 (\Omega, \mathcal{M})})
  \mapsto \int_{\Omega} a_{\tmop{ex}} \left( \frac{x}{\varepsilon} \right) |
  \nabla \tmmathbf{m} (x) |^2 \mathd x,
\end{equation}
and let us denote by $\mathcal{E}_{\mathcal{M}}$ its $\Gamma$-limit in the
metric space $(H^1 (\Omega, \mathcal{M}), d_{L^2 (\Omega, \mathcal{M})})$.
Since $H^1 (\Omega, \mathcal{M})$ is a metric subspace of $H^1 \left( \Omega,
\RR^3 \right)$, from the properties of $\Gamma$-convergence under restriction
to subspaces (see {\cite{DalMaso1993}}), we know that for every
$\tmmathbf{m}_0 \in H^1 (\Omega, \mathcal{M})$ and for every
$(\tmmathbf{m}_{\varepsilon}) \in \mathfrak{C}_{d_{L^2 (\Omega, \mathcal{M})}}
(\tmmathbf{m}_0)$
\begin{equation}
  \mathcal{E}_{\RR^3} (\tmmathbf{m}_0) \leqslant \mathcal{E}_{\mathcal{M}}
  (\tmmathbf{m}_0) \leqslant \liminf_{\varepsilon \rightarrow 0}
  \mathcal{E}_{\mathcal{M}}^{\varepsilon} (\tmmathbf{m}_{\varepsilon}) .
  \label{eq:tempupperboundgamma}
\end{equation}
Now, let $\mathcal{M}$ be a convex smooth manifold, i.e. a smooth manifold
lying on some subset of the boundary of a {\tmstrong{convex bounded domain}}
$\Theta_{\mathcal{M}}$. Next, define the {\tmstrong{tangent cone}} of
$\mathcal{M}$ at $s \in \mathcal{M}$ by the position 
\[T_s (\mathcal{M})
\assign \left\{ \tmmathbf{v} \in \RR^3 \; : \; \forall \lambda \geqslant 0, s
+ \lambda \tmmathbf{v} \notin \Theta_{\mathcal{M}} \right\},
\] 
and denote by
$\Pi_{\mathcal{M}}$ the nearest point projection on $\Theta_{\mathcal{M}}$.

Since $\Pi_{\mathcal{M}}$ is a (Lipschitz) non-expansive map,  one has
$\Pi_{\mathcal{M}} [\tmmathbf{u}] \in H^1 (\Omega, \mathcal{M})$ for every
$\tmmathbf{u} \in H^1 (\Omega, \Theta^c_{\mathcal{M}})$, moreover
(see {\cite{alouges1997new}})
\begin{equation}
  | \nabla \Pi_{\mathcal{M}} [\tmmathbf{u}] | \leqslant | \nabla \tmmathbf{u}
  | \hspace{1em} \tau \text{-a.e. in } \Omega . \label{eq:alougesobservation}
\end{equation}
Let us now suppose that $\tmmathbf{m}_0$ is sufficiently smooth so that for
every test function $\tmmathbf{m}_1 \in \mathcal{D} [\Omega ; C^{\infty}_{\#}
(Q, T_{\tmmathbf{m}_0} (\mathcal{M}))]$, the family
$\tmmathbf{m}_{\varepsilon} (x) \assign \tmmathbf{m}_0 (x) + \varepsilon
\tmmathbf{m}_1 (x, x / \varepsilon)$ belongs to $H^1 (\Omega,
\Theta^c_{\mathcal{M}})$. In this hypothesis one has 
\[\Pi_{\mathcal{M}}
[\tmmathbf{m}_{\varepsilon}] \rightarrow \tmmathbf{m}_0 \text{ \quad in } (H^1 (\Omega,
\mathcal{M}), d_{L^2 (\Omega, \mathcal{M})}).\]
 Therefore, taking into account
the estimates (\ref{eq:alougesobservation}), (\ref{eq:tempupperboundgamma}) and the fact that $\nabla_y \tmmathbf{m}_1$ is
an admissible test function (see {\cite{Allaire1992}}, Remark 1.11),  we get (passing
to the two-scale limit):
  \begin{align*}
    \mathcal{E}_{\RR^3} (\tmmathbf{m}_0) \leqslant \mathcal{E}_{\mathcal{M}}
    (\tmmathbf{m}_0) & \leqslant  \liminf_{\varepsilon \rightarrow 0}
    \mathcal{E}_{\mathcal{M}}^{\varepsilon} (\Pi_{\mathcal{M}}
    [\tmmathbf{m}_{\varepsilon}])\\
    & \leqslant  \lim_{\varepsilon \rightarrow 0} \int_{\Omega}
    a_{\tmop{ex}} \left( \frac{x}{\varepsilon} \right) \left| \nabla
    \tmmathbf{m}_0 (x) + \nabla_y \tmmathbf{m}_1 \left( x,
    \frac{x}{\varepsilon} \right) \right|^2 \mathd x\\
    & =  \int_{\Omega} \left[ \int_Q a_{\tmop{ex}} (y) | \nabla
    \tmmathbf{m}_0 (x) + \nabla_y \tmmathbf{m}_1 (x, y) |^2 \mathd y \right]
    \mathd x .
  \end{align*}
Since $\tmmathbf{m}_1\in \mathcal{D} [\Omega ; C^{\infty}_{\#} (Q, T_{\tmmathbf{m}_0} (\mathcal{M}))]$ is an arbitrary test function,
passing to the infimum we finish with the following upper and lower bound for
the manifold constrained homogenized functional:
\begin{equation}
  \mathcal{E}_{\RR^3} (\tmmathbf{m}_0) \leqslant \mathcal{E}_{\mathcal{M}}
  (\tmmathbf{m}_0) \leqslant \int_{\Omega} \Big[\inf_{\tmmathbf{\varphi} \in
  C_{\#}^{\infty} (Q_{}, T_s (\mathcal{M}))} \mathcal{I} [\xi,
  \tmmathbf{\varphi}]\Big]_{(s, \xi) \assign (\tmmathbf{m}_0 (x), \nabla
  \tmmathbf{m}_0 (x))} \mathd x, \label{eq:homheuristic}
\end{equation}
with
\begin{equation}
  \mathcal{I} [\xi, \tmmathbf{\varphi}] \assign \frac{1}{| Q |} \int_Q
  a_{\tmop{ex}} (y) | \xi + \nabla \tmmathbf{\varphi} (y) |^2 \mathd y.
\end{equation}
Since the functional $\mathcal{I} [\xi, \cdot] : C_{\#}^{\infty}
(Q_{}, T_s (\mathcal{M})) \rightarrow \RR^+$ is continuous with respect to the
$H_{\#}^1 (Q_{}, T_s (\mathcal{M}))$ norm and $C_{\#}^{\infty} (Q_{}, T_s
(\mathcal{M}))$ is dense in $H_{\#}^1 (Q_{}, T_s (\mathcal{M}))$, the infimum
in (\ref{eq:homheuristic}) can be taken over $H_{\#}^1 (Q_{}, T_s
(\mathcal{M}))$.

\begin{remark}
Quite remarkably, as we prove below, when $\mathcal{M} \assign \SSbb^2$ the
infimum appearing in the right-hand side of (\ref{eq:homheuristic}) does not
depend on the $s$-variable and coincides with the infimum in the left-hand
side of (\ref{eq:homheuristic}). Thus for $\mathcal{M} \assign \SSbb^2$ also
the lower bound to $\mathcal{E}_{\mathcal{M}}$ is sharp and
$\mathcal{E}_{\RR^3} \equiv \mathcal{E}_{\SSbb^2}$.
\end{remark}

\subsection{The tangentially homogenized Exchange Energy $\mathcal{E}_{\hom}$}

Let us go back to the application of the tangentially homogenization result in
our setting. We consider the family of exchange energy functionals, all
defined in $H^1 ( \Omega, \SSbb^2 )$, given by
$(\mathcal{E}_{\varepsilon})_{\varepsilon \in \RR^+}$. Since
$\tmmathbf{[H_1]}$ holds, Proposition \ref{prop:baba} ensures that the family
$(\mathcal{E}_{\varepsilon})_{\varepsilon \in \RR^+}$ $\Gamma$-converges in
the metric space $\left( H^1 ( \Omega ; \SSbb^2 ), d_{L^2 \left(
\Omega, \SSbb^2 \right)} \right)$, i.e$.$ with respect to the topology induces
on $H^1 ( \Omega, \SSbb^2)$ by the strong $L^2 ( \Omega,
\RR^3 )$ topology, to the functional
\begin{equation}
  \mathcal{E}_{\hom} : H^1 \left( \Omega, \SSbb^2 \right) \rightarrow \RR^+
  \hspace{1em}, \hspace{1em} \tmmathbf{m} \mapsto \mathcal{E}_{\hom}
  (\tmmathbf{m}) = \int_{\Omega} T_{} g_{\hom} (\tmmathbf{m}, \nabla
  \tmmathbf{m}) \mathd \tau,
\end{equation}
where for every $s \in \SSbb^2$ and every $\xi \in [ T_s ( \SSbb^2
) ]^3$,
\begin{equation}
  T_{} g_{\hom} (s, \xi) = \lim_{t \rightarrow \infty} \left[
  \inf_{\tmmathbf{\varphi} \in W_0^{1, \infty} \left( Q_t, T_s \left( \SSbb^2
  \right) \right)} \mathcal{I}_t [\xi, \tmmathbf{\varphi}] \right]
\end{equation}
with
\[
\mathcal{I}_t [\xi,
  \tmmathbf{\varphi}] \assign \frac{1}{| Q_t |} \int_{Q_t} a_{\tmop{ex}} (y) |
  \xi + \nabla \tmmathbf{\varphi} (y) |^2 \,\mathd y.
\]
Equivalently, since $\mathcal{I}_t [\xi, \cdot] : H^1_0 ( Q_t, T_s (
\SSbb^2 ) ) \rightarrow \RR^+$ is a continuous functional, and the
subspace $W_0^{1, \infty} ( Q_t, T_s ( \SSbb^2) )$ is
dense in $H_0^1 ( Q_t, T_s ( \SSbb^2 ) )$, we have for
every $s \in \SSbb^2$ and every $\xi \in[ T_s( \SSbb^2)]^3$
\begin{equation}
  \inf_{\tmmathbf{\varphi} \in W_0^{1, \infty} \left( Q_t, T_s \left( \SSbb^2
  \right) \right)} \mathcal{I}_t [\xi, \tmmathbf{\varphi}] =
  \inf_{\tmmathbf{\varphi} \in H_0^1 \left( Q_t, T_s \left( \SSbb^2 \right)
  \right)} \mathcal{I}_t [\xi, \tmmathbf{\varphi}],
\end{equation}
and hence
\begin{equation}
  T_{} g_{\hom} (s, \xi) = \lim_{t \rightarrow \infty} \left[
  \inf_{\tmmathbf{\varphi} \in H_0^1 \left( Q_t, T_s \left( \SSbb^2 \right)
  \right)} \mathcal{I}_t [\xi, \tmmathbf{\varphi}] \right] .
  \label{eq:tangentialhomogenized}
\end{equation}
This latter formulation is a more convenient one. Indeed, it can be easily
verified, by  {\tmname{Lax-Milgram}} theorem, that for every $t \in
\NN$, every $s \in \SSbb^2$ and every $\xi \in [ T_s ( \SSbb^2
) ]^3$, there exists a unique solution $\tmmathbf{\varphi}_t (s,
\xi) \in H_0^1 ( Q_t, T_s ( \SSbb^2 ))$ of the problem
\begin{equation}
  \tmmathbf{\varphi}_t (s, \xi) \assign \underset{\tmmathbf{\varphi} \in H_0^1
  \left( Q_t, T_s \left( \SSbb^2 \right) \right)}{\tmop{argmin}_{}}
  \mathcal{I}_t [\xi, \tmmathbf{\varphi}] . \label{eq:homoggenej}
\end{equation}
Moreover, for every $t \in \NN$ and every $s \in \SSbb^2$, the function 
\begin{equation}
\xi
\in [ T_s ( \SSbb^2 )]^3 \mapsto \tmmathbf{\varphi}_t
(s, \xi) \in W_0^{1, \infty} ( Q_t, T_s ( \SSbb^2 ) ),
\end{equation}
is a (continuous) linear map.

Let us now consider the problem that defines the {\tmstrong{classical
homogenization problem}} (see
{\cite{Braides1998,marcellini1978periodic,muller1987homogenization}}), namely
\begin{equation}
  _{} g_{\hom} (\xi) \assign \lim_{t \rightarrow \infty} \left[
  \inf_{\tmmathbf{\varphi} \in H_0^1 \left( Q_t, \RR^3 \right)} \mathcal{I}_t
  [\xi, \tmmathbf{\varphi}] \right] . \label{eq:homoggene}
\end{equation}
Our aim is now to prove that the natural extension of $g_{\hom}$ to the tangent
bundle $[ \mathcal{T} ( \SSbb^2 ) ]^3 \assign \sqcup_{s
\in \SSbb^2} [ T_s ( \SSbb^2 ) ]^3$ coincides with the
tangentially homogenized energy density $T_{} g_{\hom} (s, \xi)$. To this end
we observe that in the {\guillemotleft}classical{\guillemotright} problem
(\ref{eq:homoggene}), the function space among which the minimization takes place
is bigger than the one involved in the original problem
(\ref{eq:tangentialhomogenized}), so that $g_{\hom} (\xi) \leqslant T_{}
g_{\hom} (s, \xi)$ for every $(s, \xi) \in [ \mathcal{T} ( \SSbb^2
) ]^3$.

To prove that $g_{\hom} (\xi) \equiv T_{} g_{\hom} (s, \xi)$ on $[
\mathcal{T} ( \SSbb^2) ]^3$ it is thus sufficient to show
that for every $t \in \NN$ and for every $(s, \xi) \in [ \mathcal{T}
( \SSbb^2 ) ]^3$ there exists a function $\tmmathbf{\psi}_{t,
\xi} \in H_0^1 ( Q_t, T_s ( \SSbb^2 ))$ such that
$\mathcal{I}_t [\xi, \tmmathbf{\psi}_{t, \xi}] \leqslant g_{\hom} (\xi)$.
Having this goal in mind, we observe that if $\tmmathbf{\varphi}_{t, \xi}$
is the unique solution of the minimization problem arising in
(\ref{eq:homoggene}), denoting by $\tmmathbf{\psi}_{t, \xi} \assign
\tmmathbf{\varphi}_{t, \xi} - (\tmmathbf{\varphi}_{t, \xi} \cdot s) s$ the
nearest point projection of $\tmmathbf{\varphi}_{t, \xi}$ on $[ T_s
( \SSbb^2 ) ]^3$, one has $\tmmathbf{\psi}_{t, \xi} \in H_0^1
( Q_t, T_s ( \SSbb^2 ) )$ and
  \begin{align*}
    | \xi + \nabla \tmmathbf{\varphi}_{t, \xi} |^2 & =  | \xi + \nabla
    \tmmathbf{\psi}_{t, \xi} |^2 + | s \otimes \nabla (\tmmathbf{\varphi}_{t,
    \xi} \cdot s) |^2 + 2 (\xi + \nabla \tmmathbf{\psi}_{t, \xi}) : s \otimes
    \nabla (\tmmathbf{\varphi}_{t, \xi} \cdot s)\\
    & = | \xi + \nabla \tmmathbf{\psi}_{t, \xi} |^2 + | \nabla
    (\tmmathbf{\varphi}_{t, \xi} \cdot s) |^2 + 2 (\xi + \nabla
    \tmmathbf{\psi}_{t, \xi}) \nabla (\tmmathbf{\varphi}_{t, \xi} \cdot s)
    \cdot s\\
    & = | \xi + \nabla \tmmathbf{\psi}_{t, \xi} |^2 + | \nabla
    (\tmmathbf{\varphi}_{t, \xi} \cdot s) |^2
  \end{align*}
where, in deriving the previous equalities, we have used the relations $| a
\otimes b |^2 = | a |^2 | b |^2$, $(A : B (a \otimes b)) = A_{} b \cdot B_{}
a$, and the orthogonality relations $[\nabla \tmmathbf{\psi}_{t, \xi}]^T s =
0$, $\xi^T s = 0$. Thus we have
  \begin{align*}
    \mathcal{I}_t [\xi, \tmmathbf{\psi}_{t, \xi}] & =  \frac{1}{| Q_t |}
    \int_{Q_t} a_{\tmop{ex}} (y) | \xi + \nabla \tmmathbf{\varphi}_{t, \xi}
    (y) |^2 \mathd y \\
		& \qquad \qquad - \frac{1}{| Q_t |} \int_{Q_t} a_{\tmop{ex}} (y) | \nabla
    (\tmmathbf{\varphi}_{t, \xi} \cdot s) (y) |^2 \mathd y\\
    & =  \mathcal{I}_t [\xi, \tmmathbf{\varphi}_{t, \xi}] - \frac{1}{| Q_t
    |} \int_{Q_t} a_{\tmop{ex}} (y) | \nabla (\tmmathbf{\varphi}_{t, \xi}
    \cdot s) (y) |^2 \mathd y\\
    & \leqslant  \mathcal{I}_t [\xi, \tmmathbf{\varphi}_{t, \xi}],
  \end{align*}
and therefore, passing to the limit for $t \rightarrow \infty$ we get
$g_{\hom} (\xi) \equiv T_{} g_{\hom} (s, \xi)$. In particular $T_{} g_{\hom}
(s, \xi)$ does not depend on $s$, and is given by (\ref{eq:homoggene}).

\begin{remark}
  The fact that tangential homogenization energy density $T_{} g_{\hom} (s,
  \xi)$ reduces to the {\guillemotleft}classical{\guillemotright} one (i.e.
  $g_{\hom}$), which does not depend from the $s$-variable, is a quite
  remarkable fact. Indeed, it is possible to build elementary examples where
  the dependence on the $s$-variable in the tangential homogenization energy
  density is explicit (cfr. \tmtextup{{\cite{Babadjian2009}}}). The
  independence from the $s$-variable in our framework,  is mainly due to the
  very particular situation that for every $y \in \RR^3$, the
  Carath{\'e}odory function $g (y, \cdot) = a_{\tmop{ex}} (y) | \cdot |^2$ is
  invariant under the rotation group of the manifold under consideration (the
  2-sphere $\SSbb^2$).
\end{remark}

To complete the proof concerning the exchange energy part stated in
Theorem \ref{thm:mainresult}, it is sufficient to recall that
$g_{\hom}$ is a quadratic form in $\xi$ (see
{\cite{Braides1998,marcellini1978periodic}}), i.e. that there exists a
symmetric and positive definite matrix $A_{\hom} \in \RR^{3 \times 3}$ such
that for every $\xi \in \RR^{3 \times 3}$
\begin{equation}
  g_{\hom} (\xi) = A_{\hom} \xi : \xi . \tmcolor{red}{\label{eq:ghom}}
\end{equation}
Precisely, it is possible to show that the
{\guillemotleft}classical{\guillemotright} problem (\ref{eq:homoggene}) may be
equivalently reformulated by replacing homogeneous boundary conditions by
periodic boundary conditions (see {\cite{muller1987homogenization}} and
Appendix \ref{sec:appendix}), so that for every $(s, \xi) \in [
\mathcal{T} ( \SSbb^2) ]^3$
\begin{equation}
  T_{} g_{\hom} (s, \xi) = g_{\hom} (\xi) = \lim_{t \rightarrow \infty} \left[
  \inf_{\tmmathbf{\varphi} \in H_{\#}^1 \left( Q_t, \RR^3 \right)}
  \mathcal{I}_t [\xi, \tmmathbf{\varphi}] \right] .
\end{equation}
Furthermore, from the convexity of the integrand, it is routinely seen that we
can replace the limit for $t \rightarrow \infty$ by the computation on the
unit cell (though still with periodic boundary conditions). We are therefore
left (see {\cite{Braides1998}} and Appendix \ref{sec:appendix}) for every $(s,
\xi) \in [ \mathcal{T} ( \SSbb^2 ) ]^3$ with
\begin{equation}
  T_{} g_{\hom} (s, \xi) = g_{\hom} (\xi) = \inf_{\tmmathbf{\varphi} \in
  H_{\#}^1 \left( Q_{}, \RR^3 \right)} \int_{Q_{}} a_{\tmop{ex}} (y) | \xi +
  \nabla \tmmathbf{\varphi} (y) |^2 \mathd y . \label{eq:homogtang}
\end{equation}
Now, {\tmname{Lax-Milgram}} theorem shows that for every $\xi \in [ T_s
( \SSbb^2 ) ]^3$ there exists a unique solution
$\tmmathbf{\varphi}_{\xi}$ of this latter problem, up to an additive constant
that we may fix by requiring that $\langle \tmmathbf{\varphi}_{\xi} (y)
\rangle_Q =\tmmathbf{0}$. Moreover, the map $\xi \in [ T_s ( \SSbb^2
) ]^3 \mapsto \tmmathbf{\varphi}_{\xi} \in H_{\#}^1 ( Q,
\RR^3 )$ is a (continuous) \tmem{linear} map, and a direct computation, shows
that $A_{\hom}$ can be expressed as:
\begin{equation}
  A_{\hom} \assign \langle a_{\tmop{ex}} (y) (I + \nabla \tmmathbf{\varphi}
  (y))^T (I + \nabla \tmmathbf{\varphi} (y))  \rangle_Q \;,\;
   \tmmathbf{\varphi} \assign (\varphi_1, \varphi_2,
  \varphi_3)
\end{equation}
where for every $j \in \NN_3$ the component $\varphi_j$ is the unique (up
to a constant) solution of the {\tmem{scalar unit cell problem}}
\begin{equation}
  \varphi_j \assign \underset{\varphi \in W_{\#}^{1, \infty} \left( Q,
  \RR^3 \right)}{\tmop{argmin}}_{} \int_Q a_{\tmop{ex}} (y) [ | \nabla \varphi
  (y) + e_j |^2 ] \mathd x .
\end{equation}
\begin{remark}
  By repeating the same argument followed above to deduce the equality $T_{}
  g_{\hom} (s, \xi) = g_{\hom} (\xi)$, it is now simple to show that the
  tangentially homogenized exchange energy density can be equivalently
  characterized by the position
  \[ T_{} g_{\hom} (s, \xi) =_{} \inf_{\tmmathbf{\varphi} \in H_{\#}^1 \left(
     Q_{}, T_s \left( \SSbb^2 \right) \right)} \int_Q a_{\tmop{ex}} (y) | \xi
     + \nabla \tmmathbf{\varphi} (y) |^2 \mathd y, \]
  which is exactly the one deduced in equation (\ref{eq:homheuristic}). 
\end{remark}

\section{The periodic homogenization of the demagnetizing
field}\label{subsec:proof3}

This Section is devoted to show that the family of magnetostatic self-energies
$(\mathcal{W}_{\varepsilon})_{\varepsilon \in \RR^+}$ continuously converges
to $\mathcal{W}_{\hom}$. To this end, let us first recall some essential facts
concerning the demagnetizing field operator.

\subsection{The {\tmperson{Beppo-Levi}} space and the variational formulation
for the demagnetizing field}

From the mathematical point of view, assuming $\Omega$ to be open, bounded and
with a Lipschitz boundary, a given magnetization $\tmmathbf{m} \in L^2 (
\Omega, \RR^3 )$ generates the stray field $\tmmathbf{h}_{\mathd}
[\tmmathbf{m}] = \nabla u_{\tmmathbf{m}}$ where the potential
$u_{\tmmathbf{m}}$ solves:
\begin{equation}
  \Delta u_{\tmmathbf{m}} = - \tmop{div} (\tmmathbf{m} \chi_{\Omega}) \text{ \
  in \ } \mathcal{D}^{\prime} \left( \RR^3 \right) . \label{eq:potentialnew}
\end{equation}
In (\ref{eq:potentialnew}) we have denoted by $\tmmathbf{m} \chi_{\Omega}$
the extension of $\tmmathbf{m}$ to $\RR^3$ that vanishes outside $\Omega$. 

Once introduced
the weight $\omega (x) = (1 + | x |^2)^{- 1 / 2}$, and the weighted Lebesgue space
\[
L^2_\omega (\Omega) = \left\{ u \in \mathcal{D}^{\prime} (\Omega) \; : \; u
  \omega \in L^2 (\Omega)  \right\},\] 
	we define the 
 {\tmname{Beppo-Levi}} space
\begin{equation}
  BL^1 (\Omega) = \left\{ u \in \mathcal{D}^{\prime} (\Omega) \; : \; u
  \in L^2_\omega (\Omega) \text{ \ and \ } \nabla u \in L^2 \left( \Omega,
  \RR^3 \right) \right\},
\end{equation}
which is an Hilbert space when endowed with the scalar product $(u,v)_{B_{} L^1 ( \RR^3)} :=(\nabla u,\nabla v)_\Omega$. It is straightforward to show, by the means of {\tmname{Lax-Milgram}} theorem, the existence and uniqueness of 
the solution of the variational formulation associated to equation \eqref{eq:potentialnew}: namely to find a potential $u_{\tmmathbf{m}} \in B L^1( \RR^3 )$ such that for all $\varphi \in BL^1 ( \RR^3)$
\begin{equation}
  (u_{\tmmathbf{m}}, \varphi)_{B_{} L^1 \left( \RR^3 \right)} \assign
  \int_{\RR^3} \nabla u_{\tmmathbf{m}} \cdot \nabla \varphi \mathd \tau = -
  \int_{\RR^3} \tmmathbf{m} \cdot \nabla \varphi \mathd \tau = :
  F_{\tmmathbf{m}} [\varphi] . \label{eq:variationalformulationpotential2}
\end{equation}

Thus, for every $\tmmathbf{m} \in L^2 (\RR^3)$ there exists a unique 
$u_{\tmmathbf{m}} \in BL^1(\RR^3 )$ such that (\ref{eq:variationalformulationpotential2}) holds,
and moreover, the following stability estimate holds:
\begin{equation}
  \| u_{\tmmathbf{m}} \|_{B_{} L^1 \left( \RR^3 \right)} \equiv \| \nabla
  u_{\tmmathbf{m}} \|^2_{\Omega} \leqslant_{} \sup_{\underset{\| \nabla
  \varphi \|_{\Omega} = 1}{\varphi \in B_{} L^1 \left( \RR^3 \right)}} |
  F_{\tmmathbf{m}} [\varphi] | \leqslant \| \tmmathbf{m} \|_{L^2 \left( \RR^3
  \right)} ._{} \label{eq:stabilitylaxmilgram}
\end{equation}
The quantity $\mathbf{h}_d [\tmmathbf{m}] \assign \nabla u_{\tmmathbf{m}}$ is
what is referred to as the demagnetizing field, and it is a linear and
continuous operator from $L^2 ( \RR^3, \RR^3 )$ into $L^2 (
\RR^3, \RR^3 )$. In particular $\tmmathbf{m} \chi_{\Omega} \in L^2
( \RR^3 )$ for every $\tmmathbf{m} \in L^2 (\Omega)$ and therefore
$\mathbf{h}_d$ is also a linear and continuous operator from $L^2 (
\Omega, \RR^3 )$ into $L^2 ( \RR^3, \RR^3 )$.

\subsection{The two-scale limit of the demagnetizing field}

In this subsection we make use of the notion of two-scale convergence, to
characterize the behavior of the demagnetizing field operator under two-scale
convergence. More precisely, we suppose to have a bounded sequence
$(\tmmathbf{m}_{\varepsilon}) \in L^2 ( \RR^3 )$ which two scale
converges to some $\tmmathbf{m}_{\infty} (x, y) \in L^2$ and we want to
understand if the two-scale limit of the sequence $\mathbf{h}_d
[\tmmathbf{m}_{\varepsilon}]$ exists, and in the affirmative case to
characterize in some analytic sense such a limit. This problem has already been treated in \cite{Santugini2007} although without 
justifying the use of two-scale compactness results in weighted space.    That is why in this subsection
we start by proving two compactness results concerning two-scale convergence in the weighted spaces
$L^2_{\omega} \left( \RR^3 \right)$ and $B_{} L^1 ( \RR^3 )$. 

The first one is a {\guillemotleft}weighted{\guillemotright} variant of
the compactness result stated in Proposition \ref{prop:compactnessl2ts}, and shows that a notion of two-scale convergence in $L^2_{\omega} ( \RR^3 )$ makes sense.

\begin{proposition} \label{prop:compactl2w}
  Let $(u_{\varepsilon})$ be bounded sequence in $L^2_{\omega}( \RR^3
  )$. There exists a function $u \in L^2_{\tmop{loc}}( \RR^3
  \times Q )$ such that $\langle u \rangle_Q \in L^2_{\omega}(
  \RR^3 )$ and, up to the extraction of a subsequence,
  \begin{equation}
    \lim_{\varepsilon \rightarrow 0} \int_{\RR^3} u_{\varepsilon} (x) \varphi
    (x, x / \varepsilon_{}) \mathd x = \int_{\RR^3 \times Q} u (x, y) \varphi
    (x, y) \mathd y \mathd x
  \end{equation}
	for all $ \varphi \in \mathcal{D}[\RR^3; C^{\infty}_{\#} (Q)]$.
  In this case we say that the sequence $(u_{\varepsilon})$
  $L^2_{\omega}$-two-scale converges to $u$.
\end{proposition}

\begin{proof}
  Since $(u_{\varepsilon})$ is bounded in the Hilbert space $L^2_{\omega} (
  \RR^3)$, there exists an element $u_{\infty} \in L^2_{\omega} (
  \RR^3 )$ and a sequence $(u_{\varepsilon (n)}) \subseteq
  (u_{\varepsilon})$ such that
  \begin{equation}
    u_{\varepsilon (n)} \rightharpoonup u_{\infty} \hspace{1em} \text{weakly
    in $L^2_{\omega} ( \RR^3 )$} .
  \end{equation}
  This implies that for every bounded domain $\Omega \subseteq \RR^3$, one has
  $u_{\varepsilon (n)} \rightharpoonup u_{\infty}$ weakly in $L^2 (\Omega)$.
  We now consider a sequence of bounded domain $(\Omega_i)_{i \in \NN}$
  covering $\RR^3$. Let us start with the index $i = 1$, i.e$.$ with
  $\Omega_1$. According to the two-scale compactness result (see Proposition
  \ref{prop:compactnessl2ts}), there exists a subsequence $u_{\varepsilon
  (n_{k_1})}$ and an element $u_1 \in L^2 (\Omega_1 \times Q)$ such that
  \begin{equation}
    u_{\varepsilon (n_{k_1})} \twoheadrightarrow u_1 \text{ \ in } L^2
    (\Omega_1 \times Q) .
  \end{equation}
  Now we consider $i = 2$, i.e$.$ $\Omega_2$. Since $u_{\varepsilon (n_{k_1})}
  \rightharpoonup u_{\infty}$ weakly in $L^2_{\omega}( \RR^3 )$,
  it is possible to extract a further subsequence $( u_{\varepsilon
  (n_{k_2})})$ from $u_{\varepsilon (n_{k_1})}$ such that
  $u_{\varepsilon (n_{k_2})} \twoheadrightarrow u_2$ in $L^2 (\Omega_2 \times
  Q)$ for some suitable $u_2 \in L^2 (\Omega_2 \times Q)$. Moreover, due to
  the unicity of the two-scale limit, one has
  \begin{equation}
    u_{1 | (\Omega_1 \cap \Omega_2) \times Q \nobracket} \equiv u_{2 |
    (\Omega_1 \cap \Omega_2) \times Q \nobracket}
  \end{equation}
  whenever $\Omega_1 \cap \Omega_2 \neq \emptyset$. Proceeding in this way, we
  find for every $i \in \NN$ a subsequence $u_{\varepsilon (n_{k_i})}$ such
  that
  \begin{equation}
    n_{k_i} \subseteq n_{k_{i - 1}} \hspace{2em} \text{{{and}}}
    \hspace{2em} u_{\varepsilon (n_{k_i})} \twoheadrightarrow u_i
    \label{eq:seqindices1}
  \end{equation}
  for some $u_i \in L^2 (\Omega_i \times Q)$. We then define the
  {\tmstrong{diagonal sequence}} of indices defined by
  \begin{equation}
    n_{k_{\infty} (1)} \assign n_{k_1 (1)}, n_{k_{\infty} (2)} \assign n_{k_2
    (2)}, \ldots, n_{k_{\infty} (i)} = n_{k_i (i)}, \ldots .
  \end{equation}
  From (\ref{eq:seqindices1}) we get that for every $i \in \NN$, up to the
  first $i - 1$ terms, the sequence of indices $n_{k_{\infty}}$ is included in
  $n_{k_i}$, and this means that for every $i \in \NN$
  \begin{equation}
    u_{\varepsilon (n_{k_{\infty}})} \twoheadrightarrow u_i \text{ \ in } L^2
    (\Omega_i \times Q) .
  \end{equation}
  By observing again that $u_{i | (\Omega_i \cap \Omega_j) \times Q
  \nobracket} \equiv u_{j | (\Omega_i \cap \Omega_j) \times Q \nobracket}$ if
  $\Omega_i \cap \Omega_j \neq \emptyset$, from the {\guillemotleft}{\tmem{principe du
  recollement des morceaux{\guillemotright}}} (cfr$.$
  {\cite{schwartz1957theorie}}) there exists a unique distribution $u \in
  L^2_{\tmop{loc}} ( \RR^3 \times Q )$ such that $u_{| \Omega_i
  \nobracket \times Q} \equiv u_i$, and therefore
  \begin{equation}
    \lim_{k_{\infty} \rightarrow \infty} \int_{\RR^3} u_{\varepsilon
    (n_{k_{\infty}})} (x) \varphi (x, x / \varepsilon) \mathd x = \int_{\RR^3
    \times Q} u (x, y) \varphi (x, y) \mathd y \mathd x
  \end{equation}
  for every $\varphi \in \mathcal{D}[ \RR^3; C^{\infty}_{\#} (Q)
  ]$. Moreover since
  \begin{equation}
    u_{\varepsilon (n_{k_{\infty}})} \rightharpoonup u_{\infty} \hspace{1em}
    \text{in $L^2_{\omega} ( \RR^3)$}
	\end{equation}
		and
		\begin{equation}
		\forall i \in \NN \hspace{1em}
    u_{\varepsilon (n_{k_{\infty}})} \rightharpoonup \langle u (x, y)
    \rangle_Q \hspace{1em} \text{in $L^2 (\Omega_i)$}
  \end{equation}
  we get also $\langle u (x, y) \rangle_Q \equiv u_{\infty} \in L^2_{\omega}
 ( \RR^3 )$. This completes the proof.
\end{proof}

Exactly with the same diagonal argument, it is possible to prove the
weighted variant of the compactness result stated in Proposition
\ref{prop:compactnessSts}.

\begin{proposition}
  \label{prop:extcompactStsc}Let $(u_{\varepsilon})$ be bounded sequence in
  $BL^1 ( \RR^3 )$ weakly convergent to $u_{\infty}$. Then
  $u_{\varepsilon}$ $L^2_{\omega}$-two-scale converges to $u_{\infty} \in
  L^2_{\omega} ( \RR^3 )$ and there exists a function $v \in L^2[ \RR^3 ; H^1_{\#} (Q) / \RR]$ such that, up to the extraction of a subsequence:
  \begin{equation}
    \nabla u_{\varepsilon} \twoheadrightarrow \nabla u_{\infty} + \nabla_y v.
  \end{equation}
\end{proposition}

\begin{proof}
  We start by observing that since $u_{\varepsilon (n)} \rightharpoonup
  u_{\infty}$ weakly in $B_{} L^1 ( \RR^3 )$, $u_{\varepsilon (n)}
  \rightharpoonup u_{\infty}$ in $L^2_{\omega} ( \RR^3 )$, and
  therefore, according to the previous proposition, there exists a function $u
  \in L^2_{\omega} ( \RR^3 \times Q )$ such that, up to a
  subsequence,
  \begin{equation}
    \begin{array}{ll}
      u_{\varepsilon (n)} \twoheadrightarrow u (x, y) & \text{in }
      L^2_{\omega} ( \RR^3 \times Q )\\
      u_{\varepsilon (n)} \rightharpoonup u_{\infty} (x) \equiv \langle u (x,
      y) \rangle_Q & \text{in $L^2_{\omega} (\RR^3 )$} .
    \end{array}
  \end{equation}
  We now consider a sequence of bounded domain $(\Omega_i)_{i \in \NN}$.
  Proceeding as in the proof of the previous Proposition \ref{prop:compactl2w}, one proves that for  every $i \in \NN$ there exists a subsequence
  $u_{\varepsilon (n_{k_i})}$ such that
  \begin{equation}
    n_{k_i} \subseteq n_{k_{i - 1}} \; \text{{{ and }}}
    \; u_{\varepsilon (n_{k_i})} \twoheadrightarrow u_{\infty} (x)
    \equiv \langle u (x, y) \rangle_Q \equiv u (x, y) \hspace{1em} \text{in
    $L^2 (\Omega_i \times Q)$} . \label{eq:seqindices1Sob}
  \end{equation}
  We then define the diagonal sequence of indices defined by the position $ n_{k_{\infty} (i)} = n_{k_i (i)}$.
  From (\ref{eq:seqindices1Sob}) we get that for every $i \in \NN$, up to the
  first $i - 1$ terms, the sequence of indices $n_{k_{\infty}}$ is included in
  $n_{k_i}$, and this means that for every $i \in \NN$
  \begin{equation}
    u_{\varepsilon (n_{k_{\infty}})} \twoheadrightarrow u_{\infty} (x) \equiv
    \langle u (x, y) \rangle_Q \equiv u (x, y) \hspace{1em} \text{in $L^2
    (\Omega_i \times Q)$} .
  \end{equation}
  Thus $u \equiv u_{\infty} \in L^2_{\omega} (\RR^3 )$ in $\RR^3$.
  
  Next, we observe that since $u_{\varepsilon (n)} \rightharpoonup u_{\infty}$
  weakly in $BL^1 ( \RR^3 )$ we have $\nabla u_{\varepsilon
  (n)} \rightharpoonup \nabla u_{\infty}$ and $(\nabla u_{\varepsilon (n)})$
  bounded in $L^2 ( \RR^3 )$. Thus, according to the classical
  two-scale compactness result (see Proposition \ref{prop:compactnessSts})
  there exists a function $\tmmathbf{\kappa}_{\infty} \in L^2 ( \RR^3
  \times Q )$ such that, up to a subsequence,
  \begin{equation}
    \nabla u_{\varepsilon} \twoheadrightarrow \tmmathbf{\kappa}_{\infty}
    \hspace{1em} \text{in } L^2 ( \RR^3, \RR^3) .
  \end{equation}
  Thus, for any test function $[\varphi \otimes \tmmathbf{\psi}_{\#}] (x, y)
  \assign \varphi (x) \tmmathbf{\psi}_{\#} (y) \in \mathcal{D} [ \RR^3 ;
  C^{\infty}_{\#} (Q) ]$ with $\tmop{div}_y \tmmathbf{\psi}_{\#} (y) =
  0$ one has
  \begin{equation}
    \int_{\RR^3} u_{\varepsilon} (x) \tmop{div}_x [\varphi \otimes
    \tmmathbf{\psi}_{\#}] (x, x / \varepsilon) \mathd x = - \int_{\RR^3}
    \nabla u_{\varepsilon} (x) \cdot [\varphi \otimes \tmmathbf{\psi}_{\#}]
    (x, x / \varepsilon) \mathd x .
  \end{equation}
  Passing to the two-scale convergence on both sides we get that for a.e$.$ $x \in \RR^3$ and for every $\tmmathbf{\psi}_{\#} \in C^{\infty}_{\#} (Q)$ such that
  $\tmop{div}_y \tmmathbf{\psi}_{\#} (y) = 0$ in $Q$. 
  \begin{equation}
    \int_Q [\tmmathbf{\kappa}_{\infty} (x, y) - \nabla u_{\infty} (x)] \cdot
    \tmmathbf{\psi}_{\#} (y) \mathd y = 0 .
  \end{equation}
	Since the orthogonal
  complement of the divergence-free functions is the space of gradients, for
  a.e$.$ $x \in \RR^3$ there exists a unique function $v (x, \cdot) \in
  H^1_{\#} (Q) / \RR$ such that $\nabla_y v (x, y) \equiv
  \tmmathbf{\kappa}_{\infty} (x, y) - \nabla u_{\infty} (x)$. Thus $\nabla_y v
  \in L^2 ( \RR^3 \times Q)$ and $v (x, \cdot) \in H^1_{\#} (Q) /
  \RR$, i.e$.$ $v \in L^2 [ \RR^3, H^1_{\#} (Q) / \RR ]$. This
  completes the proof.
\end{proof}

We are now ready to prove the
two-scale convergence of the demagnetizing field operator (for this we follow the lines of \cite{Santugini2007}):

\begin{proposition}
  \label{prop:Santugininew}Let $(\tmmathbf{m}_{\varepsilon})_{\varepsilon \in
  \RR^+}$ be a bounded family in $L^2 ( \RR^3, \RR^3 )$ that
  two-scale converges to $\tmmathbf{m} (x, y)$. Then the two-scale limit of
  $(\mathbf{h}_d [\tmmathbf{m}_{\varepsilon}])_{\varepsilon \in \RR^+} \in L^2
  ( \RR^3, \RR^3)$ exists and is given by
  \begin{equation}
    \mathbf{h}_d (x, y) =\mathbf{h}_d [\langle \tmmathbf{m} (x, \cdot)
    \rangle_Q] + \nabla_y v_{\tmmathbf{m}} (x, y)
  \end{equation}
  where for every $x \in \RR^3$ the scalar function $v_{\tmmathbf{m}} (x,
  \cdot)$ is the unique solution in $H^1_{\#} (Q)$ of the cell problem
  \begin{equation}
    \Delta_y v_{\tmmathbf{m}} (x, y) = -\tmop{div}_y \tmmathbf{m} (x, y)
    \hspace{1em} \text{in } H^1_{\#} (Q) / \RR
  \end{equation}
  and therefore of the variational cell problem
  \begin{equation}
    \int_Q \tmmathbf{m} (x, y) \cdot \nabla_y \psi (y) \mathd y = - \int_Q
    \nabla_y v_{\tmmathbf{m}} (x, y) \cdot \nabla_y \psi (y) \mathd y,
    \hspace{1em} \int_Q v_{\tmmathbf{m}} (x, y) \mathd y = 0
    \label{eq:varcellpr}
  \end{equation}
  for all $\psi \in H^1_{\#} (Q)$.
\end{proposition}

\begin{proof}
  Since $(\tmmathbf{m}_{\varepsilon})$ is bounded in $L^2 ( \RR^3)$, due to the stability estimate (\ref{eq:stabilitylaxmilgram}), the
  sequence of magnetostatic potentials $(u_{\tmmathbf{m}}^{\varepsilon})$
  solution of the problem $\Delta u_{\tmmathbf{m}}^{\varepsilon} = -
  \tmop{div} (\tmmathbf{m}_{\varepsilon})$, remains bounded in $BL^1( \RR^3)$. This means that, up to a subsequence, 
  $(u_{\tmmathbf{m}}^{\varepsilon}) \rightharpoonup u^{}_{\tmmathbf{m}}$
  weakly in $B_{} L^1_{} ( \RR^3 )$ for some suitable
  $u^{}_{\tmmathbf{m}} \in BL^1 ( \RR^3 )$. Thus, according to
  Proposition \ref{prop:extcompactStsc}, there exist functions $u_{\tmmathbf{m}} \in
  BL^1 ( \RR^3 )$ and $v_{\tmmathbf{m}} \in L^2[\RR^3; H^1_{\#} (Q) / \RR ]$ such that
  \begin{equation}
    (u^{\varepsilon}_{\tmmathbf{m}}) \twoheadrightarrow u_{\tmmathbf{m}}
    \text{ \ in } L^2_{\omega} \;\;\text{} \;
    \text{{{ and }}} \;\; \nabla u_{\tmmathbf{m}}^{\varepsilon}
    (x) \twoheadrightarrow \nabla u_{\tmmathbf{m}} (x) + \nabla_y
    v_{\tmmathbf{m}} (x, y) . \label{eqs:twoscalelimits}
  \end{equation}
  In view of the previous limit relations, $u^{\varepsilon}_{\tmmathbf{m}}$ is
  expected to behave as $u_{\tmmathbf{m}} (x) + \varepsilon v_{\tmmathbf{m}}
  (x, x / \varepsilon)$. This suggest to use, in the variational formulation
  of the magnetostatic problem expressed by equation
  (\ref{eq:variationalformulationpotential2}), test functions having the form
  $\varphi (x) + \varepsilon \psi (x, x / \varepsilon)$, with $\varphi \in
  \mathcal{D} ( \RR^3 )$ and $\psi \in \mathcal{D}[ \RR^3;  C^{\infty}_{\#} (Q) ]$.
	This yields
  \begin{align*}
    & \int_{\RR^3} \nabla u_{\tmmathbf{m}}^{\varepsilon} \cdot \left( \nabla
    \varphi + \varepsilon \nabla \psi \left( x, \frac{x}{\varepsilon} \right)
    + \nabla_y \psi \left( x, \frac{x}{\varepsilon} \right) \right)_{} \mathd
    x \\
		& \qquad \qquad = - \int_{\RR^3} \tmmathbf{m}_{\varepsilon} \cdot \left( \nabla \varphi
    + \varepsilon \nabla \psi \left( x, \frac{x}{\varepsilon} \right) +
    \nabla_y \psi \left( x, \frac{x}{\varepsilon} \right) \right) \mathd x .
  \end{align*}
  From the second of the two limit relations (\ref{eqs:twoscalelimits}), we
  get
      \begin{align}
      &\int_{\RR^3 \times Q} (\nabla u_{\tmmathbf{m}} (x) + \nabla_y
      v_{\tmmathbf{m}} (x, y)) \cdot (\nabla \varphi + \nabla_y \psi (x,
      y))_{} \mathd y \mathd x \nonumber\\
       & \qquad \qquad= - \int_{\RR^3 \times Q} \tmmathbf{m} (x, y) \cdot (\nabla \varphi +
      \nabla_y \psi (x, y)) \mathd y \mathd x .
     \label{eq:choosetestfunc}
  \end{align}
  In particular, by choosing $\psi \equiv 0$ we get
  \begin{align}
      - \int_{\RR^3} \langle \tmmathbf{m}_{} (x, y) \rangle_Q \cdot \nabla
      \varphi (x) \mathd x & =  \int_{\RR^3 \times Q} (\nabla
      u_{\tmmathbf{m}} (x) + \nabla_y v_{\tmmathbf{m}} (x, y)) \cdot \nabla
      \varphi_{} \mathd y \mathd x \nonumber \\
      & =  \int_{\RR^3} \nabla u_{\tmmathbf{m}} (x) \cdot \nabla \varphi (x)
    , \label{eq:homogenizedequationmagpotential}
  \end{align}
  where the last equality follows from the fact that $\langle \nabla_y v_{\tmmathbf{m}} (x,
  y) \rangle_Q = 0$. Thus, we reach the conclusion that the weak limit
  $u_{\tmmathbf{m}}$ satisfies the variational formulation
  (\ref{eq:homogenizedequationmagpotential}), i.e$.$ is a solution of the
  {\guillemotleft}homogenized{\guillemotright} equation
  \begin{equation}
    u_{\tmmathbf{m}} (x) = - \tmop{div} \langle \tmmathbf{m} (x, y) \rangle_Q
    \hspace{1em} \text{in } BL^1_{}( \RR^3) .
  \end{equation}
  On the other hand by choosing $\varphi \equiv 0$ and $\psi (x, y) = \psi_1
  (x) \psi_2 (y)$ in (\ref{eq:choosetestfunc}) we get
  \begin{align*}
    &\int_{\RR^3} \langle (\nabla u_{\tmmathbf{m}} (x) + \nabla_y
    v_{\tmmathbf{m}} (x, y)) \cdot \nabla_y \psi_2 (y) \rangle_Q \psi_1 (x)
    \mathd x \\
		&\qquad\qquad= - \int_{\RR^3} \langle \tmmathbf{m} (x, y) \cdot \nabla_y
    \psi_2 (y) \rangle_Q \psi_1 (x) \mathd x \, ,
  \end{align*}
  and hence the so-called cell problem
  \begin{align}
      - \int_Q \tmmathbf{m} (x, y) \cdot \nabla_y \psi_2 (y) \mathd y & = 
      \int_Q (\nabla u_{\tmmathbf{m}} (x) + \nabla_y v_{\tmmathbf{m}} (x, y))
      \cdot \nabla_y \psi_2 (y) \mathd y \nonumber\\
      & =  \int_Q \nabla_y v_{\tmmathbf{m}} (x, y) \cdot \nabla_y \psi_2 (y)
      \mathd y \, ,
     \label{eq:homogenizedequationmagpotential2}
  \end{align}
  where, again, the last equality follows from the fact that $\langle \nabla_y \psi_2 (y)
  \rangle_Q = 0$. Note that the variational formulation
  (\ref{eq:homogenizedequationmagpotential2}) can be more concisely expressed
  in the form
  \begin{equation}
    \Delta_y v_{\tmmathbf{m}} (x, y) = - \tmop{div}_y \tmmathbf{m} (x, y)
    \hspace{1em} \text{in } H^1_{\#} (Q) / \RR,
  \end{equation}
  and the well-posedness of the previous variational problem is again a direct
  consequence of {\tmname{Lax-Milgram}} theorem.
\end{proof}

\subsection{The continuous limit of magnetostatic self-energy functionals
$\mathcal{W}_{\varepsilon}$}

In what follows we will make use of Proposition \ref{prop:Santugininew}, to
prove the following

\begin{proposition}
  The family of magnetostatic self-energies
  \begin{equation}
    \mathcal{W}_{\varepsilon} : \tmmathbf{m} \in L^2 (\Omega, S^2) \mapsto -
    (\tmmathbf{h}_{\mathd} [M_{\varepsilon} \tmmathbf{m}], M_{\varepsilon}
    \tmmathbf{m})_{\Omega}
  \end{equation}
  continuously converges to the functional
  \begin{equation}
    \mathcal{W}_{\hom} : \tmmathbf{m} \in L^2 (\Omega, S^2) \mapsto - \langle
    M_s \rangle_Q^2 (\tmmathbf{h}_{\mathd} [\tmmathbf{m}],
    \tmmathbf{m})_{\Omega} + \| \nabla_y v_{\tmmathbf{m}} \|_{\Omega \times
    Q}^2
  \end{equation}
  where for every $x \in \Omega$ the scalar function $v_{\tmmathbf{m}} :
  \Omega \times Q \rightarrow \mathbbm{R}$ is the unique solution of the
  following variational cell problem:
  \begin{equation}
    \tmmathbf{m} (x) \cdot \int_Q M_s (y) \nabla_y \psi (y) \mathd y = -
    \int_Q \nabla_y v_{\tmmathbf{m}} (x, y) \cdot \nabla_y \psi (y) \mathd y ,
  \end{equation}
	
	\begin{equation}
	\int_Q v_{\tmmathbf{m}} (x, y) \mathd y = 0 ,
	\end{equation}
  for all $\psi \in H^1_{\#} (Q)$.
\end{proposition}

\begin{proof}
  We know (see Proposition \ref{lemma:genRiemLeb}) that $| M_{\varepsilon}
  \tmmathbf{m} |^2 \equiv | M_{\varepsilon} |^2 \rightharpoonup \langle |
  M_{\varepsilon} |^2 \rangle_Q$ weakly$^{\ast}$ in $L^{\infty} (\Omega)$. In
  particular, by choosing $| \tmmathbf{m} |^2 \in L^1 (\Omega)$ as a test
  function we get
  \begin{align*}
    \| M_{\varepsilon} \|_{\Omega}^2 = (| M_{\varepsilon} \tmmathbf{m} |^2, |
    \tmmathbf{m} |^2)_{\Omega} \rightarrow (\langle | M_{\varepsilon} |^2
    \rangle_Q, | \tmmathbf{m} |^2)_{\Omega} &= | \Omega | \langle |
    M_{\varepsilon} |^2 \rangle_Q \\
		&= \| M_s (y) \tmmathbf{m} (x) \|_{\Omega
    \times Q}^2
  \end{align*}
  and therefore $M_{\varepsilon} (x) \tmmathbf{m} (x) \twoheadrightarrow M_s
  (y) \tmmathbf{m} (x)$ {{strongly}}. 
	
	Next, since
  $\tmmathbf{h}_{\mathd} [M_{\varepsilon} \tmmathbf{m}] \cdot M_{\varepsilon}
  \tmmathbf{m}$ is bounded in $L^2 (\Omega)$, from Proposition
  \ref{Prop:2scalproduct} and Proposition \ref{prop:Santugininew}, we get
  \begin{align}
    & \lim_{\varepsilon \rightarrow 0} - (\tmmathbf{h}_{\mathd} [M_{\varepsilon}
    \tmmathbf{m}], M_{\varepsilon} \tmmathbf{m})_{\Omega}  \nonumber\\
		& \qquad\qquad= - \langle M_s
    \rangle_Q^2 (\tmmathbf{h}_{\mathd} [\tmmathbf{m}], \tmmathbf{m})_{\Omega}
    - \int_{\Omega \times Q} \nabla_y v_{\tmmathbf{m}} (x, y) \cdot M_s (y)
    \tmmathbf{m} (x) \mathd x \mathd y.
  \end{align}
  Now, we observe that for every $x \in \Omega$ the scalar function
  $v_{\tmmathbf{m}} (x, \cdot)$ is the unique solution of the variational cell
  problem (\ref{eq:varcellpr}), therefore setting $\psi (\cdot) \assign
  v_{\tmmathbf{m}} (x, \cdot)$ in (\ref{eq:varcellpr}) we get
  \begin{equation}
    -\tmmathbf{m} (x) \cdot \int_Q M_s (y) \nabla_y v_{\tmmathbf{m}} (x, y)
    \mathd y = \int_Q | \nabla_y v_{\tmmathbf{m}} (x, y) |^2 \mathd y
  \end{equation}
  and therefore
  \begin{equation}
    \lim_{\varepsilon \rightarrow 0} - (\tmmathbf{h}_{\mathd} [M_{\varepsilon}
    \tmmathbf{m}], M_{\varepsilon} \tmmathbf{m})_{\Omega} = - \langle M_s
    \rangle_Q^2 (\tmmathbf{h}_{\mathd} [\tmmathbf{m}], \tmmathbf{m})_{\Omega}
    + \| \nabla_y v_{\tmmathbf{m}} \|_{\Omega \times Q}^2 = :
    \mathcal{W}_{\hom} (\tmmathbf{m}) . \label{eq:wcdemag}
  \end{equation}
  Now we show that the family $\mathcal{W}_{\varepsilon^{}}$ continuously
  converges to $\mathcal{W}_{\hom}$. This amounts to prove that for every
  $\tmmathbf{m} \in L^2 ( \Omega, \SSbb^2 )$ and every $\varepsilon
  \in \RR^+$, $\varepsilon < \varepsilon_0$ and $\|
  \tmmathbf{m}-\tmmathbf{m}_0 \|_{\Omega} < \delta$ implies $|
  \mathcal{W}_{\varepsilon} (\tmmathbf{m}) -\mathcal{W}_{\hom}
  (\tmmathbf{m}_0)  | < \eta$. To this end, for every $\tmmathbf{m},
  \tmmathbf{m}_0 \in L^2 \left( \Omega, \SSbb^2 \right)$ we split:
  \begin{equation}
    | \mathcal{W}_{\varepsilon} (\tmmathbf{m}) -\mathcal{W}_{\hom}
    (\tmmathbf{m}_0) | \leqslant | \mathcal{W}_{\varepsilon} (\tmmathbf{m})
    -\mathcal{W}_{\hom} (\tmmathbf{m}_{}) | + | \mathcal{W}_{\hom}
    (\tmmathbf{m}) -\mathcal{W}_{\hom} (\tmmathbf{m}_0) | .
  \end{equation}
  From (\ref{eq:wcdemag}) it follows the existence of a sufficiently small
  $\varepsilon_0$ such that
  \begin{equation}
    \forall \varepsilon < \varepsilon_0 \hspace{1em} |
    \mathcal{W}_{\varepsilon} (\tmmathbf{m}) -\mathcal{W}_{\hom}
    (\tmmathbf{m}_{}) | < \frac{\eta}{2} . \label{eq:magnselfenergyboundhom1}
  \end{equation}
  On the other hand, since $\tmmathbf{m} \mapsto \nabla_y v_{\tmmathbf{m}}$ is
  a linear and continuous map from $L^2 \left( \Omega, \RR^3 \right)$ into
  $L^2 \left( \Omega \times Q, \RR^3 \right)$ with $\| \nabla_y
  v_{\tmmathbf{m}} \|_{\Omega \times Q} \leqslant c_M  \| \tmmathbf{m}
  \|_{\Omega}$ (and $c_M \assign \| M_s \|_Q$) one has
  \begin{equation}
	\def\arraystretch{1.3}
    \begin{array}{lll}
      | \mathcal{W}_{\hom} (\tmmathbf{m}) -\mathcal{W}_{\hom} (\tmmathbf{m}_0)
      | & \leqslant & \langle M_s \rangle_Q^2 | (\tmmathbf{h}_{\mathd}
      [\tmmathbf{m}], \tmmathbf{m})_{\Omega} - (\tmmathbf{h}_{\mathd}
      [\tmmathbf{m}_0], \tmmathbf{m}_0)_{\Omega} |\\
      &  & \hspace{1em} + | \| \nabla_y v_{\tmmathbf{m}} \|^2_{\Omega \times
      Q} - \| \nabla_y v_{\tmmathbf{m}_0} \|^2_{\Omega \times Q}  |\\
      & \leqslant & \langle M_s \rangle_Q^2 | (\tmmathbf{h}_{\mathd}
      [\tmmathbf{m}+\tmmathbf{m}_0], \tmmathbf{m}-\tmmathbf{m}_0)_{\Omega} |\\
      &  & \hspace{1em} + | (\nabla_y v_{\tmmathbf{m}+\tmmathbf{m}_0},
      \nabla_y v_{\tmmathbf{m}-\tmmathbf{m}_0})_{\Omega \times Q} |\\
      & \leqslant & 2 | \Omega |^{1 / 2} (\langle M_s \rangle_Q^2 + c_M^2) \|
      \tmmathbf{m}-\tmmathbf{m}_0 \|_{\Omega} .
    \end{array} \label{eq:magnselfenergyboundhom2}
  \end{equation}
  and the previous estimate together with (\ref{eq:magnselfenergyboundhom1})
  clearly concludes the proof.
\end{proof}

\section{The homogenized anisotropy and interaction
energies}\label{subsec:proof4}

This section is devoted to the proof of the continuous convergence of the
family of anisotropy energy functionals $\mathcal{A}_{\varepsilon}$ and of the
family of interaction energy functionals $\mathcal{Z}_{\varepsilon}$,
respectively to $\mathcal{A}_{\hom}$ and $\mathcal{Z}_{\hom}$, whose
expression is given by (\ref{eq:homAnthm1}) and (\ref{eq:homIntthm1}).

\subsection{The continuous limit of the anisotropy energy functionals
$\mathcal{A}_{\varepsilon}$}

\begin{proposition}
  If the anisotropy energy density $\varphi_{\tmop{an}} : \RR^3 \times \SSbb^2
  \rightarrow \RR^+$ is $Q$-periodic with respect to the first variable and
  globally lipschitz with respect to the second one (uniformly with respect to
  the first variable) then the family $\mathcal{A}_{\varepsilon}$ of
  anisotropy energies continuously converges to the homogenized anisotropy
  energy
  \begin{equation}
    \mathcal{A}_{\hom} : \tmmathbf{m} \in L^2 \left( \Omega, \SSbb^2 \right)
    \mapsto \int_{\Omega \times Q} \varphi_{\tmop{an}} (y, \tmmathbf{m} (x))
    \mathd y \mathd x .
  \end{equation}
\end{proposition}

\begin{proof}
  Again, we have to prove that for every $\tmmathbf{m}_0 \in L^2 (
  \Omega, \SSbb^2 )$ one has 
	\[
	\lim_{(\varepsilon, \tmmathbf{m})
  \rightarrow (0, \tmmathbf{m}_0)} \mathcal{A}_{\varepsilon} (\tmmathbf{m})
  =\mathcal{A}_{\hom} (\tmmathbf{m}_0).
	\]
	For every $\tmmathbf{m},
  \tmmathbf{m}_0 \in L^2 \left( \Omega, \SSbb^2 \right)$ we split
  \begin{align}
      & \int_{\Omega \times Q} \varphi (x / \varepsilon, \tmmathbf{m} (x)) -
      \varphi (y, \tmmathbf{m}_0 (x)) \mathd y \mathd x \nonumber\\
			 & \qquad\qquad =  \int_{\Omega
      \times Q} \varphi (x / \varepsilon, \tmmathbf{m} (x)) - \varphi (y,
      \tmmathbf{m} (x))_Q \mathd y \mathd x \nonumber\\
			& \qquad\qquad\qquad+ \int_{\Omega \times Q} \varphi (y,
      \tmmathbf{m} (x)) - \varphi (y, \tmmathbf{m}_0 (x)) \mathd y \mathd x .
      \label{eq:homan0}
  \end{align}
  Next, we observe that since $\varphi (x / \varepsilon, \tmmathbf{m} (x))
  \rightharpoonup \langle \varphi (y, \tmmathbf{m} (x)) \rangle_Q$
  weakly$^{\ast}$ in $L^{\infty} (\Omega)$ (cfr. Lemma
  \ref{lemma:genRiemLeb}), there exists a sufficiently small $\varepsilon_0$
  such that for every $\varepsilon < \varepsilon_0$
  \begin{equation}
    \hspace{1em} \left|  \int_{\Omega \times Q} \varphi (x / \varepsilon,
    \tmmathbf{m} (x)) - \varphi (y, \tmmathbf{m} (x)) \mathd y \mathd x
    \right| < \frac{\eta}{2} . \label{eq:homan1}
  \end{equation}
  On the other hand, by the global lipschitz continuity in the second variable
  of $\varphi_{\tmop{an}}$, we have
  \begin{align}
    \int_{\Omega \times Q} | \varphi (y, \tmmathbf{m} (x)) - \varphi (y,
    \tmmathbf{m}_0 (x)) | \mathd y \mathd x & \leqslant c_L  \int_{\Omega} |
    \tmmathbf{m} (x) -\tmmathbf{m}_0 (x) | \mathd x \\
		& \leqslant c_L | \Omega
    |^{1 / 2} \| \tmmathbf{m}-\tmmathbf{m}_0 \|_{\Omega} . \label{eq:homan2}
  \end{align}
  Substituting estimates (\ref{eq:homan1}) and (\ref{eq:homan2}) into
  (\ref{eq:homan0}) we get
  \begin{equation}
    | \mathcal{A}_{\varepsilon} (\tmmathbf{m}) -\mathcal{A}_{\hom}
    (\tmmathbf{m}_0)  | \leqslant \frac{\eta}{2} \noplus + c^{\star} \|
    \tmmathbf{m}-\tmmathbf{m}_0 \|_{\Omega} \hspace{2em} \left( \text{with }
    c^{\star} \assign c_L | \Omega |^{1 / 2} \right) .
  \end{equation}
  Therefore for every $\tmmathbf{m} \in L^2 ( \Omega, \SSbb^2 )$
  such that $\| \tmmathbf{m}-\tmmathbf{m}_0 \|_{\Omega} < \eta / (2
  c^{\star})$ we get $| \mathcal{A}_{\varepsilon} (\tmmathbf{m})
  -\mathcal{A}_{\hom} (\tmmathbf{m}_0)  | \leqslant \eta$, and this concludes
  the proof.
\end{proof}

\begin{corollary}
  {\tmem{$(\nobracket$Uniaxial anisotropy energy density$\nobracket)$}}. If
  $\varphi (y, \tmmathbf{m}) = \kappa (y) | \tmmathbf{m} (x) \wedge
  \tmmathbf{u} (y) |^2$ then
  \[ \mathcal{A}_{\hom} (\tmmathbf{m}) = \int_{\Omega} \langle \kappa
     \rangle_Q - \langle \kappa \tmmathbf{u} \otimes \tmmathbf{u} \rangle_Q :
     \tmmathbf{m} \otimes \tmmathbf{m} \mathd \tau . \]
\end{corollary}

\subsection{The continuous limit of interaction energy functionals
$\mathcal{Z}_{\varepsilon}$}\label{subsec:proof5}

The convergence of $(\mathcal{Z}_{\varepsilon})_{\varepsilon \in \RR^+}$ to
$\mathcal{Z}_{\varepsilon}$ is straightforward. Indeed this energy term is
expressed by the product, with respect to the $L^2 (\Omega)$ scalar product,
of the constant function $\tmmathbf{h}_a$ and the weakly converging sequence
$(M_{\varepsilon} \tmmathbf{m})_{\varepsilon \in \RR^+} \rightharpoonup
\langle M_s \rangle_Q$ weakly$^{\ast}$ in $L^{\infty} (\Omega)$ (cfr. Lemma
\ref{lemma:genRiemLeb}). Therefore repeating the same argument given in the
previous subsection:
\[ \mathcal{Z}_{\varepsilon}  \underset{\varepsilon \rightarrow
   0}{\xrightarrow{\Gamma_{\tmop{cont}}}} \mathcal{Z}_{\hom} \hspace{1em}
   \text{with} \hspace{1em} \mathcal{Z}_{\hom} (\tmmathbf{m}) \assign -
   \langle M_s \rangle_Q \int_{\Omega} \tmmathbf{h}_a \cdot \tmmathbf{m}
   \mathd {\tau}. \]
\section{Proof of \tmtextup{Theorem
\ref{thm:mainresult}} completed}\label{sec:proof}

It is now easy to complete the proof of Theorem \ref{thm:mainresult}. 
Indeed the equicoercivity of the family of \tmtextsc{Gibbs-Landau}
free energy functionals
$(\mathcal{G}^{\varepsilon}_{\mathcal{L}})_{\varepsilon \in \RR^+}$ expressed
by (\ref{eq:GLComposites}) has been proved in Section \ref{subsec:proof1}. It
is therefore sufficient to recall the stability properties of the
$\Gamma$-limit under the sum of a continuously convergent family of
functionals. In fact, what has been proved in the previous subsections, can be
summarized by the following convergence scheme
\begin{equation}
  \mathcal{E}_{\varepsilon}  \underset{\varepsilon \rightarrow
  0}{\xrightarrow{\Gamma_{}}} \mathcal{E}_{\hom} \hspace{1em}, \hspace{1em}
  \mathcal{W}_{\varepsilon}  \underset{\varepsilon \rightarrow
  0}{\xrightarrow{\Gamma_{\tmop{cont}}}} \mathcal{W}_{\hom} \hspace{1em},
  \hspace{1em} \mathcal{A}_{\varepsilon}  \underset{\varepsilon \rightarrow
  0}{\xrightarrow{\Gamma_{\tmop{cont}}}} \mathcal{A}_{\hom} \hspace{1em},
  \hspace{1em} \mathcal{Z}_{\varepsilon}  \underset{\varepsilon \rightarrow
  0}{\xrightarrow{\Gamma_{\tmop{cont}}}} \mathcal{Z}_{\hom} .
\end{equation}
Thus, Proposition \ref{prop:sumGlimit} completes the proof.

\section{Conclusions and acknowledgment}
We have given in this paper a complete theory for periodic microstructured magnetic
materials. Obtained through a process of $\Gamma-$convergence the model derives 
rigorously the energy terms from the parameters of each constituent of the
sample and the mixing geometry of the different materials in the unit periodic cell. 
We believe that the result applies to most of magnetic composites that are nowadays
considered, e.g. those obtained from a mixing of hard and soft phases \cite{Hadjipanayis,Skomski}
or the multilayer magnetic materials \cite{Hu,Renard}. In this latter case, the formula obtained 
furthermore simplifies since the exchange coefficient can be analytically computed.
We leave the exploration of potential applications to a forthcoming work.
 
\appendix\section{Restatement of some well-known result}\label{sec:appendix}

In this appendix we state, in a bit general form, two results that we
mentioned in the previous sections. We start with a result whose proof can be
extracted from {\cite{muller1987homogenization}}, which permits to pass, in
the characterization of the homogenized energy density, from homogeneous
boundary conditions to periodic ones. The proof, although well-known in the
homogenization community, is difficult to find in the classical monographs on
the subject:

\begin{proposition}
  Let $\mathcal{V}$ be a vector subspace of $\RR^3$ and let $g : \RR^3 \times
  \RR^{3 \times 3} \rightarrow \RR^+$ be a Carath{\'e}odory function
  satisfying the hypotheses of {\tmname{(thm)}} with $\mathcal{M} \assign
  \mathcal{V}$ \tmtextup{(}\tmtextup{see Proposition}
  \tmtextup{\ref{prop:baba}}\tmtextup{)}. Then, for every $(s, \xi) \in
  \mathcal{V} \times \mathcal{V}^3$ one has $T_{} g_{\hom} (s, \xi) =
  T^{\#}_{} g_{\hom} (\xi)$ with
  \begin{equation}
    T^{\#}_{} g_{\hom} (\xi) \assign \lim_{t \rightarrow \infty} \left[
    \inf_{\tmmathbf{\varphi} \in H_{\#}^1 (Q_t, \mathcal{V})} \frac{1}{| Q_t
    |} \int_{Q_t} g (y, \xi + \nabla \tmmathbf{\varphi} (y)) \mathd y \right]
    . \label{eq:homogeneoustoperiodic}
  \end{equation}
\end{proposition}

\begin{proof}
  Since $T_s (\mathcal{V}) \equiv \mathcal{V}$ it is obvious that $T_{}
  g_{\hom}$ does not depend on the $s$-variable. Moreover, since $H_0^1 (Q_t,
  \mathcal{V}) \subseteq H_{\#}^1 (Q_t, \mathcal{V})$ one trivially has
  $T^{\#}_{} g_{\hom} (\xi) \leqslant T_{} g_{\hom} (s, \xi)$. Let us
  therefore focus on the converse inequality. For any $\tmmathbf{\varphi} \in
  H_0^1 (Q_t, \mathcal{V})$ and any $\xi \in \mathcal{V}^3$ we consider the
  family $(\tmmathbf{m}_{\varepsilon} (x) \assign \xi x + \varepsilon
  \tmmathbf{\varphi} (x / \varepsilon)) \in H^1 (\Omega, \mathcal{V})$ which
  converges to $\xi x$ in the metric space $(H^1 (\Omega, \mathcal{V}), d_{L^2
  (\Omega, \mathcal{V})})$. By the characterizing properties of the
  $\Gamma$-limit we have:
  \begin{equation}
    T_{} g_{\hom} (s, \xi) = \frac{1}{| \Omega |} \mathcal{E}_{\hom} (\xi x)
    \leqslant \frac{1}{| \Omega |} \liminf_{\varepsilon \rightarrow 0}
    \mathcal{F}_{\varepsilon} (\tmmathbf{m}_{\varepsilon}) = \lim_{\varepsilon
    \rightarrow 0}  \frac{1}{| \Omega |} \int_{\Omega} g \left(
    \frac{x}{\varepsilon}, \xi + \nabla \tmmathbf{\varphi} \left(
    \frac{x}{\varepsilon} \right) \right) \mathd x .
  \end{equation}
  Since the integrand in the last member of the previous equation is
  $Q_t$-periodic and satisfies the hypotheses of the generalized
  {\tmname{Riemann-Lebesgue}} Lemma (cfr. Proposition \ref{lemma:genRiemLeb})
  we finish with the estimate
  \begin{equation}
    T_{} g_{\hom} (s, \xi) \leqslant \frac{1}{| Q_t |} \int_{Q_t} g (y, \xi +
    \nabla \tmmathbf{\varphi} (y)) \mathd y \hspace{1em} \forall \xi \in
    \mathcal{V}^3 \nocomma, \forall \tmmathbf{\varphi} \in H_0^1 (Q_t,
    \mathcal{V}) .
  \end{equation}
  Passing to the infimum we get the stated result.
\end{proof}

Finally, in the hope to achieve our goal of making the material contained in
this paper accessible to as large audience as possible, we prefer to state --
again in a bit general form -- and give a more
{\guillemotleft}commented{\guillemotright} proof of the following well-known
result (cfr$.$ {\cite{Braides1998}}) which permits, in the convex case, to
pass from variations that are periodic over an ensemble of $t^3$ 1-cells to
variations that are periodic in the unit cell $Q$.

\begin{proposition}
  Let $g : \RR^3 \times \RR^{3 \times 3} \rightarrow \RR^+$ be a
  Carath{\'e}odory integrand satisfying the hypotheses of {\tmname{(thm)}}
  \tmtextup{(}\tmtextup{Proposition} \tmtextup{\ref{prop:baba}}\tmtextup{)}.
  If for any $y \in \RR^3$ the function $g (y, \cdot)$ is convex, then for any
  convex subset $\mathcal{C}$ of $\RR^3$ and every $t \in \NN$ the following
  equality holds:
  \begin{equation}
    \inf_{\tmmathbf{\varphi} \in H_{\#}^1 (Q_t, \mathcal{C})} \frac{1}{| Q_t
    |} \int_{Q_t} g (y, \xi + \nabla \tmmathbf{\varphi} (y)) \mathd y =
    \inf_{\tmmathbf{\varphi} \in H_{\#}^1 (Q_{}, \mathcal{C})} \int_Q g (y,
    \xi + \nabla \tmmathbf{\varphi} (y)) \mathd y.
    \label{eq:convexhom2equality}
  \end{equation}
\end{proposition}

\begin{proof}
  One of the main ingredients of the argument, is in observing that for every
  $t \in \NN$ the following decomposition of the $t$-cube in $1$-cubes holds:
  \begin{equation}
    Q_t = \underset{(i_1, i_2, i_3) \in \NN_t^3}{{\sqcup}} Q_1^{(i_1,
    i_2, i_3)}  \hspace{1em} \text{{\tmstrong{with}}} \hspace{1em} Q_1^{(i_1,
    i_2, i_3)} \assign Q_1 + e_{(i_1, i_2, i_3)}
  \end{equation}
  and $e_{(i_1, i_2, i_3)} \assign (i_1 - 1) e_1 + (i_2 - 1) e_2 + (i_3 - 1)
  e_3$. Thus, from the $1$-periodicity of $g$ in the first variable, we arrive
  to the intuitive and well-known fundamental equality:
  \begin{align}
      \int_{Q_t} g (y, \xi + \nabla \tmmathbf{\varphi} (y)) \mathd y & =
      \sum_{(i_1, i_2, i_3) \in \NN_t^3} \int_{Q_1^{(i_1, i_2, i_3)}} g (y,
      \xi + \nabla \tmmathbf{\varphi} (y)) \mathd y \nonumber \\
      & = \sum_{(i_1, i_2, i_3) \in \NN_t^3} \int_{Q_1} g (y, \xi + \nabla
      \tmmathbf{\varphi} (y)) \mathd y \nonumber \\
      & =   | Q_t | \int_{Q_1} g (y, \xi + \nabla \tmmathbf{\varphi} (y))
      \mathd y .
   \label{eq:convexhom1}
  \end{align}
  Next, we observe that $H_{\#}^1 (Q_{}, \mathcal{C}) \subseteq H_{\#}^1 (Q_t,
  \mathcal{C})$ for every $t \in \NN$, and therefore, from the previous
  equality we get:
  \begin{align*}
      \inf_{\tmmathbf{\varphi} \in H_{\#}^1 (Q_t, \mathcal{C})} \frac{1}{| Q_t
      |} \int_{Q_t} g (y, \xi + \nabla \tmmathbf{\varphi} (y)) \mathd y &
      \leqslant  \inf_{\tmmathbf{\varphi} \in H_{\#}^1 (Q, \mathcal{C})}
      \frac{1}{| Q_t |} \int_{Q_t} g (y, \xi + \nabla \tmmathbf{\varphi} (y))
      \mathd y\\
      & = \inf_{\tmmathbf{\varphi} \in H_{\#}^1 (Q, \mathcal{C})} \int_Q g
      (y, \xi + \nabla \tmmathbf{\varphi} (y)) \mathd y .
  \end{align*}
  To complete the proof we now prove that the previous inequality is actually
  an equality, and therefore (in particular) that the left hand side of
  (\ref{eq:convexhom2equality}) does not depend on $t \in \NN$. To this end it
  is sufficient to show that for every $t \in \NN$ and every $\tmmathbf{\psi}
  \in H_{\#}^1 (Q_t, \mathcal{C})$ there exists a $\tmmathbf{\varphi} \in
  H_{\#}^1 (Q, \mathcal{C})$ such that
  \begin{equation}
    \int_Q g (y, \xi + \nabla \tmmathbf{\varphi} (y)) \mathd y \leqslant
    \frac{1}{| Q_t |} \int_{Q_t} g (y, \xi + \nabla \tmmathbf{\psi} (y))
    \mathd y.
  \end{equation}
  With this goal in mind, for every $\tmmathbf{\psi} \in H_{\#}^1 (Q_t,
  \mathcal{C})$, we define the convex combination
  \begin{equation}
    \tmmathbf{\varphi} (y) \assign \frac{1}{| Q_t |} \sum_{(i_1, i_2, i_3) \in
    \NN_t^3} \tmmathbf{\psi} (y + e_{(i_1, i_2, i_3)}) .
  \end{equation}
  Clearly $\tmmathbf{\varphi}$ belongs to $H_{\#}^1 (Q, \mathcal{C})$.
  Moreover, by the convexity and $1$-periodicity of $g$ with respect to the
  first variable, from equality (\ref{eq:convexhom1}), we get:
  \begin{align*}
      \int_Q g (y, \xi + \nabla \tmmathbf{\varphi} (y)) \mathd y & = 
      \frac{1}{| Q_t |} \int_{Q_t} g (y, \xi + \nabla \tmmathbf{\varphi} (y))
      \mathd y\\
      & \leqslant  \frac{1}{| Q_t |} \sum_{(i_1, i_2, i_3) \in \NN_t^3}
      \frac{1}{| Q_t |} \int_{Q_t} g (y, \xi + \nabla \tmmathbf{\psi} (y +
      e_{(i_1, i_2, i_3)})) \mathd y\\
      & =  \frac{1}{| Q_t |} \sum_{(i_1, i_2, i_3) \in \NN_t^3} \frac{1}{|
      Q_t |} \int_{Q_t} g (y + e_{(i_1, i_2, i_3)}, \xi + \nabla
      \tmmathbf{\psi} (y + e_{(i_1, i_2, i_3)})) \mathd y\\
      & = \frac{1}{| Q_t |} \sum_{(i_1, i_2, i_3) \in \NN_t^3} \frac{1}{|
      Q_t |} \int_{Q_t} g (y, \xi + \nabla \tmmathbf{\psi} (y)) \mathd y\\
      & =  \frac{1}{| Q_t |} \int_{Q_t} g (y, \xi + \nabla \tmmathbf{\psi}
      (y)) \mathd y,
    \end{align*}
  and this completes the proof.
\end{proof}

\end{document}